\documentclass[hidelinks,onefignum,onetabnum]{siamart220329}

\allowdisplaybreaks
\usepackage{amssymb}
\usepackage{bm}

\usepackage{subfigure}

\usepackage{lipsum}
\usepackage{amsfonts}
\usepackage{graphicx}
\usepackage{epstopdf}
\usepackage{algorithmic}
\ifpdf
  \DeclareGraphicsExtensions{.eps,.pdf,.png,.jpg}
\else
  \DeclareGraphicsExtensions{.eps}
\fi

\newsiamremark{remark}{Remark}
\newsiamremark{hypothesis}{Hypothesis}
\crefname{hypothesis}{Hypothesis}{Hypotheses}
\newsiamthm{claim}{Claim}

\headers{Linear FEM for biharmonic problem on surfaces}{Y. Cai, H. Guo, and Z. Zhang}

\title{Continuous Linear  Finite Element Method for Biharmonic Problems on Surfaces}

\author{Ying Cai\thanks{Beijing Computational Science Research Center, Beijing 100193, China (\email{ycai@csrc.ac.cn}).}
\and Hailong Guo\thanks{School of Mathematics and Statistics, The University of Melbourne, Parkville, VIC 3010, Australia (\email{hailong.guo@unimelb.edu.au}).}
\and Zhimin Zhang\thanks{Department of Mathematics, Wayne State University, Detroit, MI 48202, USA (\email{ag7761@wayne.edu}).}
}

\usepackage{amsopn}

\newcommand{\jp}[1]{
[\![#1]\!]}
\newcommand{\avg}[1]{
\{\!\!\{#1\}\!\!\}}
\renewcommand{\div}{\mathrm{div}}
\newcommand{\dif}{\,\mathrm{d}}
\newcommand{\norm}[1]{|\!|\!|#1|\!|\!|}
\newcommand{\ele}{\mathrm{e}}
\NewDocumentCommand{\dgalext}{m}{%
  \sbox0{%
    \mathsurround=0pt 
    $\left\{\vphantom{#1}\right.\kern-\nulldelimiterspace$%
  }%
  \sbox2{\{}%
  \ifdim\ht0=\ht2
    \{\kern-.45\wd2 \{#1\}\kern-.45\wd2 \}%
  \else
  \fi
}

\NewDocumentCommand{\dgalx}{om}{%
  \sbox0{\mathsurround=0pt$#1\{$}%
  \sbox2{\{}%
  \ifdim\ht0=\ht2
    \{\kern-.45\wd2 \{#2\}\kern-.45\wd2 \}%
  \else
    \mathopen{#1\{\kern-.5\wd0 #1\{}
    #2
    \mathclose{#1\}\kern-.5\wd0 #1\}}
  \fi
}

\usepackage{subcaption}

\usepackage{multirow}

\graphicspath{{fig/}{fig/mesh/}}

\begin{document}

\maketitle

\begin{abstract}
This paper presents an innovative continuous linear finite element approach to effectively solve biharmonic problems on surfaces. The key idea behind this method lies in the strategic utilization of a surface gradient recovery operator to compute the second-order surface derivative of a piecewise continuous linear function defined on the approximate surface, as conventional notions of second-order derivatives are not directly applicable in this context. By incorporating appropriate stabilizations, we rigorously establish the stability of the proposed formulation. Despite the presence of geometric error, we provide optimal error estimates in both the energy norm and $L^2$ norm. Theoretical results are supported by numerical experiments.
\end{abstract}



\section{Introduction}\label{sec: intro}
Fourth-order elliptic partial differential equations (PDEs) on surfaces find widespread application in engineering and physical domains, encompassing problems such as the surface Cahn–Hilliard equation \cite{DuJuTian2011,elliott2015evolving}, Kirchhoff plate equation on surfaces \cite{walker2022kirchhoff}, surface Navier-Stokes equations \cite{brandnernumerical}, and biomembranes \cite{bonito2011dynamics, elliott2010modeling}. While substantial research has been devoted to finite element methods for second-order PDEs on surfaces, following the influential work of Dziuk \cite{dziuk1988finite} as reviewed in \cite{dziuk2013finite}, there exists a noticeable gap in the literature concerning the development and analysis of finite element methods tailored specifically for fourth-order PDEs on surfaces. To date, only a limited number of papers have been dedicated to addressing the challenges associated with developing and analyzing finite element methods for fourth-order PDEs on surfaces.

The principal challenge hindering the finite element discretization of fourth-order PDEs on surfaces is the general $C^{0,1}$ continuity of the approximate surface. Consequently, only $H^1$ elements on the approximate surface are well-defined, making them unsuitable for the direct discretization of fourth-order differential operators. Du et al. \cite{DuJuTian2011} proposed a mixed finite element method for discretizing bilaplacian operator for Cahn-Hilliard equation on surface. Larsson and Larsson \cite{larsson2017continuous}  addressed this issue by proposing a  continuous/discontinuous Galerkin method for surface biharmonic problems. They employed continuous, piecewise quadratic polynomials on the approximate surface as the finite element space and enforced weak continuity of the normal by introducing penalty terms. Elliott et al. proposed the second-order splitting scheme for numerically solving fourth-order PDEs on surfaces \cite{elliott2020splitconstrain, elliott2019secondorder}.   Another approach presented by Walker \cite{walker2022kirchhoff, SW2022} is the surface Hellan–Herrmann–Johnson method, which is specifically designed for solving the Kirchhoff plate equation on surfaces using a mixed formulation. This method avoids the use of $C^1$ finite elements. The identical technique was employed in \cite{WSW2024} for approximating the surface Hessian of a discrete scalar function.  

In the existing literature, various well-known gradient recovery methods have been proposed and successfully applied in scientific computations for planar domains. Notably, these methods include the Weighted Average (WA) methods, Superconvergence Patch Recovery (SPR) introduced by Zienkiewicz and Zhu \cite{Zienkiewicz}, and the Polynomial Preserving Recovery (PPR) technique as presented in works by Zhang and Naga \cite{naga2004posteriori, naga2005polynomial}. Subsequently, efforts have been made to generalize these recovery methods for surfaces, as demonstrated in the works by Du and Ju \cite{DuJu2005},  Wei et al. \cite{wei2010superconvergence} and Dong and Guo \cite{dong2020parametric}.

Over the last decade, the gradient recovery technique for solving fourth-order PDEs on planar domains has attracted considerable attention, with several studies examining this approach, including Guo et al.  \cite{chenguozhangzou2017,guo2018ac}, Lamichhane  \cite{lamichhane2011stabilized, lamichhane2014finite}, and Droniou et al. \cite{droniou2019hessian}. The appeal of this technique lies in its utilization of the $P_1$ conforming finite element space as the approximation space, leading to reduced computational costs compared to many classical methods. A significant challenge in solving fourth-order problems stems from the fact that the variational formulation often involves second derivatives, causing linear functions to vanish upon double differentiation. The key aspect of gradient recovery methods, particularly for fourth-order problems, revolves around reconstructing numerical gradient to obtain a continuous piecewise vector that better approximates the true solution's gradient and can be differentiated again. This crucial step ensures accurate and reliable solutions to fourth-order PDEs on planar domains.

In this paper, we extend the method introduced in Guo et al. \cite{guo2018ac}  to discretize biharmonic equations on surfaces, and we emphasize that this extension is nontrivial. Firstly, in Guo et al. \cite{guo2018ac}, a discrete Poincaré inequality is postulated to ensure the stability of the scheme, and its verification is limited to nearly structured meshes. However, this assumption becomes unrealistic when dealing with geometrically complicated surfaces, where generating a structured mesh is not feasible. Secondly, the gradient recovery procedure is conducted on the approximate surface rather than the exact surface. This distinction leads to the possibility of non-coplanar conormal vectors for two adjacent triangles deduced from the approximate surface. Consequently, when constructing the discrete formulation, we must consider the jump in the conormal vectors across the common edge of two elements. Furthermore, the estimation of geometric errors between the exact surface and its approximation significantly impacts the accuracy of the error estimates.

There are three major contributions to the  recovery based surface finite element method for   approximating biharmonic problems on surfaces in this work.
First, we remove the assumption of discrete Poincar\'e inequality, required in \cite{guo2018ac} to achieved the stability of the scheme,
 by adding a residual type stabilization to enrich the numerical gradient in the weak formulation. Meanwhile, to overcome the difficulty caused
 by the discontinuity of the conormal vector of the edge and execute the $H^2$ continuity, some consistency terms defined on edges
 are added to the discrete bilinear form to control the adverse effects from those discontinuities,
and then  the robust stability result is obtained. Second, it is pertinent to highlight that recovered gradients typically lack
direct approximation properties with respect to exact gradients, rendering the error estimate challenging.
In this regard, we build a weak approximation of the gradient recovery operator, and complement it with the standard geometric error analysis,
along the same lines that presented in \cite{dziuk1988finite}. Consequently, we ultimately derive an optimal a priori error estimate in energy norm.
Third, the violation of Galerkin orthogonality, arising from the consistency between the exact gradient and the recovered gradient, as well as geometric errors between the exact surface and the discrete surface, presents challenges in directly applying Aubin-Nitsche's technique to obtain $L^2$ error estimates. To address this, we draw inspiration from the methodology presented in \cite{shylaja2020improved} to analyze the consistency error between the continuous bilinear form and its discrete counterpart, which plays a central role in employing Aubin-Nitsche's technique. Additionally, to achieve the desired convergence rate of $\mathcal O(h^2)$ and account for the presence of geometric errors, traditional geometric error estimates are insufficient. To overcome this limitation, we invoke the non-standard geometric error estimate, the $\mathbf P_h\mathbf n$ lemma established in \cite{larsson2017continuous}. Leveraging this approach, we successfully demonstrate an optimal error estimate in the $L^2$ norm.

The remainder of the paper is structured as follows. In Section \ref{sec:pre}, we provide an introduction to the background materials concerning tangential differential operators, the model biharmonic problem on surfaces, and discrete surface representation.
In Section \ref{sec:fem}, we present the discrete formulation based on the gradient recovery method for approximating biharmonic equations on surfaces.
In Section \ref{sec:stab}, we demonstrate the stability of the discrete problem.
The optimal error estimates in the energy norm and $L^2$ norm are discussed in Section \ref{sec:err}, where we analyze the accuracy of the proposed recovery based surface finite element  method for biharmonic problems on surfaces.
We present numerical results in Section \ref{sec:num}.
Finally, in Section \ref{sec:con}, we draw some conclusions.

\section{Notation and setting}
\label{sec:pre}

In this section, we aim to provide the necessary preliminary knowledge regarding the (discrete) surface and the model problems. The first subsection presents a list of relevant notation used for surfaces. The second subsection introduces the model problem. Finally, the third subsection focuses on the discrete surface and its associated approximation properties.
\subsection{Surfaces and differential operators on surfaces}
Let $S$ be a compact, oriented, and $C^4$ smooth surface in $\mathbb{R}^3$ with $\partial S=\varnothing$. Considering a constant $\delta>0$, we define the open strip neighborhood $U_\delta:= \{x\in \mathbb{R}^3| \mathrm{dist}(x,S) < \delta\}$ surrounding $S$.  The oriented distance function to $S$, denoted by $d(x)$, is assumed to be well-defined in $U_\delta$ when $\delta$ is sufficiently small. For clarity, we impose $d(x)>0$ for points outside the surface $S$, $d(x)<0$ for points inside the surface $S$, and $d(x)=0$ for points on $S$. Here, $\nabla$ represents the  gradient operator in Euclidean space. The outward-pointing normal vector of $S$ is given by $\mathbf{n}(x)=\nabla d(x)$, and the Weingarten map is denoted by $\mathbf{H}(x)=\nabla^2d(x)$. It is worth mentioning that $\mathbf H$ and $\mathbf n$ are defined for any point $x\in U_\delta$. According to Gilbarg and Trudinger \cite{gilbarg1998elliptic}, there exists a sufficiently small positive constant $\delta_0$ such that we can define the closest point mapping $p: U_{\delta}\to S$ as follows:
\[p(x)=x-d(x)\mathbf{n}\circ p(x) \quad \forall x\in U_{\delta},\]
where $\delta<\delta_0$.
For any point $x\in S$, the principal curvatures of $S$ at $x$ are represented by $\kappa_1(x)$ and $\kappa_2(x)$.

Let $\mathbf{P}=\mathbf I-\mathbf{n}\otimes\mathbf{n}$ denote the projection onto the tangent plane of $S$, with $\otimes$ representing the usual tensor product. For any scalar function $v$ defined on $S$, we can extend $v$ to $C^1(U_\delta)$ and the function remains notated as $v$ after the extension. The tangential gradient of $v$ on $S$ is given by
\[\nabla_S v:=\mathbf{P}\nabla v=\nabla v-(\nabla v\cdot\mathbf{n})\mathbf{n}.\]
Similarly, we define the surface divergence operator $\div_S\bm{w}$ for a vector field $\bm{w}$ defined on $S$ as
\[\div_S\bm{w}:=\nabla_S\cdot\bm{w}=\nabla\bm{w}-\mathbf{n}^t\nabla \bm{w}\mathbf{n},\]
and the Laplace-Beltrami operator $\Delta_S$ on $S$ is given by
\[\Delta_Sv:=\div_S\nabla_Sv=\Delta v-(\nabla v\cdot \mathbf{n})(\nabla\cdot \mathbf{n})-\mathbf{n}^t\nabla^2v\mathbf{n}.\]

For any non-negative integer $r$, the surface Sobolev space on the set $\Sigma$ is defined as:
\[H^r(\Sigma)=\{v\in L^2(\Sigma):\nabla_\Sigma^\alpha v\in L^2(\Sigma)\quad\forall |\alpha|\leq r\},\]
equipped with the norm and seminorm given by:
\[\|v\|_{r,\Sigma}=\left(\sum_{|\alpha|\leq r}\|\nabla_\Sigma^\alpha v\|_{0,\Sigma}^2\right)^{\frac12}, \quad
|v|_{r,\Sigma}=\left(\sum_{|\alpha|=r}\|\nabla_\Sigma^\alpha v\|_{0,\Sigma}^2\right)^{\frac12}.\]
Furthermore, let $(\cdot, \cdot)_{\Sigma}$ denote the $L^2$ inner product on $\Sigma$.

The Green's formula  for tangential differential operators, as stated in Larsson and Larsson \cite{larsson2017continuous}, is given by:
\begin{equation}\label{equ:green}
	(\div_S\bm{v},w)_S=-(\bm{v},\nabla_S w)_S+(\mathrm{tr}(\mathbf{H})\mathbf{n}\cdot \bm{v},w)_S.
\end{equation}

  Throughout this paper, we use shorthand notations $x\lesssim y$ and $x\gtrsim y$  to mean that $x\leq Cy$ and $x\geq Cy$ for a genetic constant $C$ which independent of the mesh size. $x\thickapprox y$ is for the statement $x\lesssim y$ and $y\lesssim x$.

  \subsection{Model equations}
\label{ssec:model}
In this paper, we consider the biharmonic problem stated as follows:
\begin{equation}\label{eq:problem}
   \Delta_S^2u=f \quad \text{in } S,
\end{equation}
where $f$ is the load function satisfying the compatibility condition $\int_Sf\dif\sigma=0$, with $\mathrm{d}\sigma$ representing the surface measure. To ensure the uniqueness of the solution, we impose the constraint $\int_S u\dif\sigma=0$.

By applying Green's formula \eqref{equ:green} twice, the weak formulation of problem \eqref{eq:problem} is defined as follows: Find $u\in H^2(S)/\mathbb R$ such that
\begin{equation}\label{eq:weak}
 a(u,v)=l(v) \quad \forall v\in H^2(S)/\mathbb R,
\end{equation}
where the bilinear form is
\begin{equation}\label{equ:bilinear_a}
	a(u,v)=(\Delta_Su,\Delta_Sv)_S
\end{equation}
and
the linear functional is
\begin{equation}\label{equ:lfun}
	l(v)=(f,v)_S.
\end{equation}
Since the Laplace-Beltrami equation
\[\Delta_Su=0\quad\text{ in }S\]
only has zero solution in the space  $H^1(S)/\mathbb R$, we know  that the null space of $\Delta_S$ on $S$ consists only of constant functions. Therefore, by the Lax-Milgram Lemma \cite{evans2010}, problem \eqref{eq:weak} is well-posed. Moreover, the solution of \eqref{eq:weak} satisfies the following elliptic regularity result (cf. \cite{besse2007einstein, larsson2017continuous}):
\begin{equation}\label{inq:regularity}
  \|u\|_{4,S}\lesssim \|f\|_{0,S}.
\end{equation}

\subsection{Discrete surface}
We assume that the surface $S$ is approximated by a polyhedron $S_h\subset U_\delta$, which consists of a union of triangular faces. The associated triangulation is denoted as $\mathcal{T}_h$, with the mesh size given by $h=\max\{\mathrm{diam}\,\tau: \tau \in \mathcal{T}_h\}$.
Let $\mathcal{N}_h$ and $\mathcal{E}_h$ denote the sets of all vertices and edges of the triangulation, respectively. We consider a quasi-uniform mesh, where all vertices lie on the surface $S$. For each edge $E\in\mathcal{E}_h$, there exist two elements $\tau^+$ and $\tau^-$ such that $E=\partial\tau^+\cap\partial\tau^-$. We denote the unit conormal vector of $E$ corresponding to $\tau^+$ as $\mathbf{n}_E^+$, which represents the unit outward normal vector of $E$ in the tangent plane of $\tau^+$. Similarly, we define the unit conormal vector $\mathbf{n}_E^-$ for $\tau^-$. It should be noted that unlike in the case of a mesh on a planar domain, the vectors $\mathbf{n}_E^+$ and $\mathbf{n}_E^-$ are typically not coplanar, as illustrated in Figure \ref{fig:conormal}.
\begin{figure}
\centering
    \subfigure{
        \includegraphics[width=1.5in]{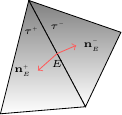}
    }
      \subfigure{
        \includegraphics[width=1.8in]{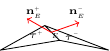}
    }
\caption{Two elements and their conormal vectors on $E$ in a different view.}
\label{fig:conormal}
\end{figure}
For a function $w$ defined on $S_h$,
the jump and average of $w$ across the edge $E$ can be defined as follows:
\[\jp{w}:=\lim_{\varepsilon\to0^+}\Big(w(x-\varepsilon \mathbf n_E^+)-w(x-\varepsilon \mathbf n_E^-)\Big),\]
\[\avg{w}:=\lim_{\varepsilon\to0^+}\frac12\Big( w(x-\varepsilon\mathbf n_E^+)+
 w(x-\varepsilon\mathbf n_E^-)\Big).\]

Let $\mathbf n_h$ be the unit outer normal vector of $S_h$.
At the discrete level, we can define the discrete projection $\mathbf{P}_h: U_{\delta}\to S_h$ in a similar manner to the continuous case, given by
\[\mathbf{P}_h=\mathbf I-\mathbf n_h\otimes \mathbf n_h.\]
Similarly, we can define the gradient and divergence operators on $S_h$ as
\[\nabla_{S_h}v=\mathbf{P}_h\nabla v\quad \text{and} \quad \div_{S_h}\bm{w}=\nabla_{S_h}\cdot\bm{w},\]
respectively. Furthermore, we define the Laplace-Beltrami operator on $S_h$ as
\[\Delta_{S_h}v=\div_{S_h}\nabla_{S_h}v.\]

We establish crucial relationships between functions defined on the surfaces $S$ and $S_h$ to facilitate our subsequent analysis. Let $v$ be a function defined on $S$, inducing a unique function $v^\ele$ defined on $S_h$ through the mapping:
  \[v^\ele(x)=v(p(x))\qquad\forall x\in S_h.\]
Conversely, for a function $v_h$ defined on $S_h$, we define the lifting of $v_h$ to $S$ as:
\[v_h^\ell(x)=v_h(\xi(x))\qquad\forall x\in S,\]where $\xi(x)$ is the unique solution to the equation $x=p(\xi(x))$, i.e.,
\[x=\xi-d(\xi)\mathbf n(x).\]

In particular, we can lift any element $\tau \in \mathcal{T}_h$ to the exact surface $S$ by defining the surface element $\tau^\ell = p(\tau)$. Consequently, the collection $\mathcal T_h^\ell = \cup_{\tau \in \mathcal T_h}\{\tau^\ell\}$ constitutes a conforming triangulation of $S$.

Similarly, for any edge $E \in \mathcal E_h$, we denote its lifted counterpart on the exact surface $S$ as $E^\ell = p(E)$, and we designate $\mathbf n_{E^\ell}^\pm$ to represent the conormals of $E^\ell$.


We can recall the transformation of the gradient on the surface between functions and lifted functions. By applying the chain rule, we have the following relation \cite{demlow2007adaptive, dziuk1988finite}:
\begin{equation}\label{equ:transform}
\nabla_{S_h}v_h(x)=\mathbf P_h(\mathbf I-d\mathbf H)\mathbf P\nabla_Sv_h^\ell(x), \quad x\in S_h.
\end{equation}

Let $\dif\sigma_h$ be the surface measure on $S_h$, then we have $\dif\sigma=\mu_h(x)\dif\sigma_h$, where (see \cite{demlow2007adaptive})
\[\mu_h(x)=(1-d(x)\kappa_1^\ele(x))(1-d(x)\kappa_2^\ele(x))\mathbf{n}\cdot\mathbf{n}_h, \quad x\in S_h,\]
and \[\kappa_i^\ele(x)=\frac{\kappa_i(p(x))}{1+d(x)\kappa_i(p(x))}, \quad i=1,2.\]
Here, $\kappa_i(p(x))$ denotes the principal curvatures at the closest point $p(x)$ on the surface.

We also define
\[\mathbf R_h=\frac{1}{\mu_h}\mathbf P(\mathbf I-d\mathbf H)\mathbf P_h(\mathbf I-d\mathbf H)\mathbf P.\]
As a result, we have the following well-known relation based on \eqref{equ:transform}:
\begin{equation}\label{equ:transformbigrad}
(\nabla_{S_h}u_h,\nabla_{S_h}v_h)_{S_h}=(\mathbf{R}_h\nabla_Su_h^\ell,\nabla_Sv_h^\ell)_S.
\end{equation}
Correspondingly, for $E\in\mathcal{E}_h$, let $\dif s_h$ and $\dif s$ be the curve measures corresponding to $E$ and $E^\ell$.

In the subsequent lemma, we aggregate several geometric error estimates that establish the relationship between $S_h$ and $S$.

\begin{lemma}
The following estimates hold provided that the mesh size $h$ is sufficiently small:
\begin{align}
\label{geo1} &\|d\|_{L^\infty(S_h)}\lesssim h^2,\\
\label{geo2} &\|\mathbf n-\mathbf n_h\|_{L^\infty(S_h)}\lesssim h,\\
\label{geo3} &\|\mathbf P-\mathbf P_h\|_{L^\infty(S_h)}\lesssim h,\\
\label{geo5} &\|1-\mu_h\|_{L^{\infty}(S_h)}\lesssim h^2,\\
\label{geo7} &\|1-\mathbf n\cdot\mathbf n_h\|_{L^\infty(S_h)}\lesssim h^2,\\
\label{geo4} &\|(\mathbf R_h-\mathbf I)\mathbf P\|_{L^\infty(S)}\lesssim h^2,\\
\label{geo6} &\|\mathbf n_{e^\ell}^\pm-\mathbf P \mathbf n_e^\pm\|_{L^\infty(E^\ell)}\lesssim h^2 \quad\forall E\in \mathcal E_h.
\end{align}

\begin{proof}
The proof can be found in literature such as \cite{dziuk1988finite, dedner2013analysis, burman2015stabilized, larsson2017continuous}.
\end{proof}
\end{lemma}

Note that the extension $v^\ele$ of a smooth function $v$ defined on $S$ is only piecewise smooth on $S_h$. For the purpose of analysis, we employ the Sobolev spaces of piecewise smooth functions, which will be utilized in our proof.

For a given positive integer $r$, we define the broken Sobolev space as follows:
\[H_h^r(S_h):=\{\varphi: \varphi|_\tau\in H^r(\tau)\quad\forall\tau \in \mathcal T_h\}\]
equipped with the norm and seminorm:
\[\|v\|_{r,S_h}=\left(\sum_{\tau\in\mathcal T_h}\|v\|_{r,\tau}^2\right)^\frac12,\qquad
|v|_{r,S_h}=\left(\sum_{\tau\in\mathcal T_h}|v|_{r,\tau}^2\right)^\frac12.\]

Regarding the relationship between a function defined on the exact surface and its lifting on the discrete surface, we recall the following norm equivalence, cf. \cite{dziuk1988finite, burman2019finite, larsson2017continuous}.
\begin{lemma}\label{lem:equivalence}
For any $v\in H^k(S)$ with $k\geq1$, we have, for sufficiently small $h$, the following inequalities hold:
\begin{align}\label{123}
& \|v^\ele\|_{0,S_h}\lesssim \|v\|_{0,S}\lesssim \|v^\ele\|_{0,S_h},\\
& |v^\ele|_{1,S_h}\lesssim |v|_{1,S_h}\lesssim |v^\ele|_{1,S_h},\\
& |v|_{k,S}\lesssim \|v^\ele\|_{k,S_h},\\
& |v^\ele|_{k,S_h}\lesssim \|v\|_{k,S}.
\end{align}
\end{lemma}

Similarly, we recall the norm equivalence on edges, as shown in \cite{burman2019finite}.
\begin{lemma}\label{lem:equivalenceedge}
Let $E\in \mathcal E_h$. For $v\in H^1(E)$, the following inequalities hold if $h$ is small sufficiently:
\begin{align}\label{inq:equivalenceedgeL2}
&\|v^\ell\|_{0,E^\ell}\lesssim\|v\|_{0,E}\lesssim \|v^\ell\|_{0,E^\ell}, \\
\label{inq:equivalenceedgeH1}
&|v^\ell|_{1,E^\ell}\lesssim|v|_{1,E}\lesssim |v^\ell|_{1,E^\ell}.
\end{align}
\end{lemma}

\section{A recovery based linear surface finite element method}
\label{sec:fem}
The continuous linear finite element space $\mathcal V_h$ on the discrete surface $S_h$ is defined as
\begin{equation}\label{eq:femspace}
	\mathcal{V}_h = \{v\in C^0(S_h): v|_{\tau} \in \mathbb{P}_1(\tau) \quad \forall \tau\in \mathcal{T}_h \}.
\end{equation}
Let $\{\phi_P(x): P\in \mathcal{N}_h \}$ denote the nodal basis of $\mathcal{V}_h$. Consequently, $\mathcal{V}_h = \text{span}\{\phi_P, P\in\mathcal N_h\}$. The gradient recovery operator $G_h: \mathcal V_h\to \mathcal V_h \times \mathcal V_h \times \mathcal V_h$ is then defined as follows: For a given function $v_h\in \mathcal V_h$ and node $P\in \mathcal N_h$, the value of $G_hv_h$ at $P$ is
\begin{equation} \label{equ:wa}
	G_hv_h(P)=\frac{1}{|\Omega_P|}\int_{\Omega_P}\nabla_{S_h}v_h\dif\sigma_h,
\end{equation}
where the element patch $\Omega_P=\bigcup\limits_{P\in \bar \tau,\tau\in\mathcal T_h}\tau$. The recovered function $G_hv_h$ can be expressed as
\[G_hv_h=\sum_{P\in\mathcal N_h}G_hv_h(P)\phi_P.\]

Suppose the mesh is sufficiently regular, and it is shown in \cite{wei2010superconvergence} that the weighted averaging gradient recovery operator $G_h$ fulfills the $\mathcal{O}(h^2)$ consistency property as follows:
\begin{equation}\label{equ:consistency}
  \|\nabla_S u-(G_hu_I)^\ell\|_{0,S}\lesssim h^2\|u\|_{3,S},
\end{equation}
for all $u\in H^3(S)$, where $u_I=I_hu$ represents the nodal basis function interpolation of $u$ to the space $\mathcal V_h$.

\begin{remark}\label{rem:gr}
The selection of weighted averaging methods is motivated by theoretical considerations. However, for practical purposes, the parametric polynomial preserving recovery technique \cite{dong2020parametric} is preferred due to its proven $\mathcal{O}(h^2)$ consistency on arbitrary meshes. Notably, it is well-established that on a uniform mesh, the weighted averaging method and the polynomial preserving recovery method for planar domains are equivalent.
\end{remark}

Drawing on the idea of $C^0$ interior penalty methods (\cite{EGHLMT2002, BS2005, larsson2017continuous}), we define the bilinear form $a_h(\cdot,\cdot)$ on $\mathcal V_h\times \mathcal V_h$ associated with the  recovery based surface finite element method for the biharmonic problem \eqref{eq:weak} as follows:
\begin{equation}\label{equ:bilinear}
	\begin{split}
		&a_h(u_h,v_h)\\
		 :=& (\div_{S_h}G_hu_h,\div_{S_h}G_hv_h)_{S_h} - \sum_{e\in\mathcal E_h} (\avg{\div_{S_h}G_hu_h},\jp{G_hv_h\cdot\mathbf n_E})_E \\
& - \sum_{E\in\mathcal E_h} (\avg{\div_{S_h}G_hv_h},\jp{G_hu_h\cdot\mathbf n_E})_E + \frac\gamma h\sum_{E\in\mathcal E_h} (\jp{G_hu_h\cdot\mathbf n_E},\jp{G_hv_h\cdot\mathbf n_E})_E \\
& + \gamma_{\text{stab}}\mathfrak S_h(u_h,v_h),
	\end{split}
\end{equation}
where the stabilization term $\mathfrak S_h$ is given by
\[\mathfrak S_h(u_h,v_h) := (\nabla_{S_h}u_h-G_hu_h,\nabla_{S_h}v_h-G_hv_h)_{S_h},\]
and $\gamma,\gamma_{\text{stab}}>0$  are parameters.

Let $V_h=\mathcal V_h/\mathbb R$ denote the linear element space for solving problem \eqref{eq:weak}. The discrete problem is formulated as follows: Find $u_h\in V_h$ such that
\begin{equation}\label{equ:weak}
a_h(u_h,v_h) = l_h(v_h) \quad \forall v_h\in V_h,
\end{equation}
where the discete linear functional
\begin{equation}\label{equ:discrete_lfun}
	l_h(v_h)=(f_h,v_h)_{S_h}
\end{equation}
 with
$f_h=f^\ele-\frac{1}{|S_h|}\int_{S_h}f^\ele\dif\sigma_h$ as in \cite{dziuk1988finite}.
\begin{remark} \label{rem:stab}
It is worth noting that the residual stabilization $\mathfrak S_h$ is included in the bilinear form $a_h$ to improve the stability of discrete systems. This is necessary because we can only generally assert that
\[\|G_hv_h\|_{0,S}\lesssim |v_h|_{1,S} \quad \forall v_h\in V_h,\]
which does not guarantee the coercivity of $a_h$ on $V_h$. In \cite{guo2018ac}, a discrete Poincar\'e inequality
\[\|v_h\|_{0,S}\lesssim \|G_hv_h\|_{0,S}\]
is established for the case where the mesh is almost structured, and for any $v_h\in\mathcal V_h$ with zero trace on the boundary of the planar domain, coercivity follows. However, in general, it is challenging to generate a structured mesh for surfaces, especially when the surface is geometrically complex.
\end{remark}

\begin{remark}
In this paper, we only consider the biharmonic equation on the surface without a boundary. We would like to remark that our method can also be generalized to the surface with a boundary straightforwardly using Nitsche's method \cite{Ni1971, burman2015stabilized}.
\end{remark}

\section{Stability analysis}
\label{sec:stab}

Based on the bilinear form $a_h(\cdot, \cdot)$, we define the following mesh-depedent  energy semi-norm
\begin{equation}\label{equ:norm}
	\begin{split}
	\norm{v_h}:=\Big(\|\div_{S_h}G_hv_h\|_{0,S_h}^2&+\sum_{E\in\mathcal E_h}h\|\avg{\div_{S_h}G_hv_h}\|_{0,E}^2\\
&+\sum_{E\in\mathcal E_h}h^{-1}\|\jp{G_hv_h\cdot\mathbf n_E}\|_{0,E}^2+\mathfrak S_h(v_h,v_h)\Big)^{\frac12},	
	\end{split}
\end{equation}
and we shall show that $\norm{\cdot}$ is a norm on $V_h$.

In our subsequent analysis, we will frequently use the following trace inequality and inverse inequality. For $\tau\in \mathcal T_h$, and $v\in H^1(\tau)$, $v_h\in \mathcal V_h$, we have
\begin{equation}
    \|v\|_{0,\partial\tau}^2\lesssim h^{-1}\|v\|_{0,\tau}^2+h|v|_{1,\tau}^2,
\end{equation}
and
\begin{equation}
    |v_h|_{1,\tau}\lesssim h^{-1}\|v_h\|_{0,\tau}.
\end{equation}

We begin with  establishing the interpolation of the discrete $H^1$ norm in terms of the $L^2$ norm and the discrete $H^2$ norm.

\begin{lemma}\label{lem:norm1}
For any $v_h\in V_h$, it holds:
\begin{equation*}
  	\|G_hv_h\|_{0,S_h}^2\lesssim \Big(\|\div_{S_h}G_hv_h\|_{0,S_h}+\sum_{E\in\mathcal E_h}h^{-\frac12}\|\jp{G_hv_h\cdot\mathbf n_E}\|_{0,E}\Big)
  \|v_h\|_{0,S_h}.
\end{equation*}
\end{lemma}

\begin{proof}
For $v_h\in V_h$, considering the fact that $\nabla_{S_h}v_h$ is a piecewise constant vector over $\mathcal T_h$, we obtain:
\begin{equation}\label{inq:norm11}
  	\begin{split}
  		\|G_hv_h\|_{0,S_h}^2&\thickapprox\frac13\sum_{P\in\mathcal N_h}|\Omega_P||G_hv_h(P)|^2
    =\frac13\sum_{P\in\mathcal N_h}G_hv_h(P)\cdot\int_{\Omega_P}\nabla_{S_h}v_h\dif\sigma_h\\
    &=\sum_{P\in\mathcal N_h}G_hv_h(P)\cdot\int_{S_h}\nabla_{S_h}v_h\phi_p\dif\sigma_h
    =\int_{S_h}G_hv_h\cdot\nabla_{S_h}v_h\dif\sigma_h.
  	\end{split}
\end{equation}
Using integration by parts and the Cauchy-Schwarz inequality for the right-hand side of \eqref{inq:norm11}, we obtain
\begin{equation}\label{inq:norm12}
  	\begin{split}
  		&\int_{S_h}G_hv_h\cdot\nabla_{S_h}v_h\dif\sigma_h\\
=&-\int_{S_h}\div_{S_h}G_hv_h\cdot v_h\dif \sigma_h+\sum_{E\in\mathcal E_h}\int_E
v_h\cdot\jp{G_hv_h\cdot\mathbf n_E}\dif s_h\\
\leq& \|\div_{S_h}G_hv_h\|_{0,S_h}\|v_h\|_{0,S_h}+\sum_{E\in\mathcal E_h}h^{\frac12}\|v_h\|_{0,E}h^{-\frac12}\|\jp{G_hv_h
\cdot\mathbf n_E}\|_{0,E}.
  	\end{split}
\end{equation}

We deduce from the Cauchy-Schwarz inequality, the trace inequality, and the inverse inequality that:
\begin{equation}\label{inq:norm13}
  	\begin{split}
    &\sum_{E\in\mathcal E_h}h^{\frac12}\|v_h\|_{0,E}h^{-\frac12}\|\jp{G_hv_h\cdot\mathbf
 n_E}\|_{0,E} \\
\leq& \Big(\sum_{E\in\mathcal E_h}h\|v_h\|_{0,E}^2\Big)^{\frac12}
   \Big(\sum_{E\in\mathcal E_h}h^{-1}\|\jp{G_hv_h\cdot\mathbf n_E}\|_{0,E}^2\Big)^{\frac12}\\
\lesssim & \|v_h\|_{0,S_h}\Big(\sum_{E\in\mathcal E_h}h^{-1}\|\jp{G_hv_h\cdot\mathbf n_E}\|_{0,E}^2\Big)^{\frac12}.
  	\end{split}
\end{equation}
Therefore, combining \eqref{inq:norm11}, \eqref{inq:norm12}, and \eqref{inq:norm13}, we arrive at the desired result.
\end{proof}

We can utilize the previous lemma to demonstrate that the semi-norm defined in equation \eqref{equ:norm} is actually a norm.

\begin{lemma}\label{lem:norm}
For any $v_h\in V_h$, we have:
\begin{equation}\label{inq:norm}
\|v_h\|_{1,S_h}+\|G_hv_h\|_{0,S_h}\lesssim \norm{v_h}.
\end{equation}
Furthermore, $\norm{\cdot}$ is a norm on $V_h$.
\end{lemma}

\begin{proof}
Note that $\int_{S_h}v_h\dif\sigma_h=0$. From the Poincar\'e inequality (see \cite{BrSc2008, Ciarlet2002}), we have:
\begin{equation}\label{inq:norm10}
  \|v_h\|_{0,S_h}\lesssim \|\nabla_{S_h}v_h\|_{0,S_h}.
\end{equation}
The definition of the energy semi-norm \eqref{equ:norm}, the Poincar\'e inequality \eqref{inq:norm10}, Lemma \ref{lem:norm1} and the inverse triangle inequality imply:
\begin{equation}\label{inq:norm14}
    	\begin{split}
    		&\norm{v_h} \\
    		\gtrsim &\|\div_{S_h}G_hv_h\|_{0,S_h}+\sum_{E\in\mathcal E_h}h^{-\frac12}\|\jp{G_hv_h\cdot\mathbf n_E}\|_{0,E}
+\|\nabla_{S_h} v_h-G_hv_h\|_{0,S_h}\\
\gtrsim & \frac{\|G_hv_h\|_{0,S_h}^2}{\|\nabla_{S_h} v_h\|_{0,S_h}}
+\|\nabla_{S_h} v_h-G_hv_h\|_{0,S_h}\\
= & \|\nabla_{S_h} v_h\|_{0,S_h}\left(\left\|\frac{G_hv_h}{\|\nabla_{S_h} v_h\|_{0,S_h}}\right\|_{0,S_h}^2
+\left\|\frac{\nabla_{S_h} v_h}{\|\nabla_{S_h} v_h\|_{0,S_h}}-\frac{G_hv_h}{\|\nabla_{S_h} v_h\|_{0,S_h}}\right\|_{0,S_h}
\right)\\
\geq &\|\nabla_{S_h} v_h\|_{0,S_h}\left(\left\|\frac{G_hv_h}{\|\nabla_{S_h} v_h\|_{0,S_h}}\right\|_{0,S_h}^2
-\left\|\frac{G_hv_h}{\|\nabla_{S_h} v_h\|_{0,S_h}}\right\|_{0,S_h}+1\right)\\
\geq &\frac34\|\nabla_{S_h} v_h\|_{0,S_h}\gtrsim\|v_h\|_{1,S_h}.
    	\end{split}
    \end{equation}
Similarly, utilizing Lemma \ref{lem:norm1}, it holds:
\begin{equation}\label{inq:norm15}
 	\begin{split}
 	\norm{v_h}
&\gtrsim \frac{\|G_hv_h\|_{0,S_h}^2}{\|\nabla_{S_h} v_h\|_{0,S_h}}
+\|\nabla_{S_h} v_h-G_hv_h\|_{0,S_h}\\
&\geq \|G_hv_h\|_{0,S_h}\left(\frac{\|G_hv_h\|_{0,S_h}}{\|\nabla_{S_h} v_h\|_{0,S_h}}
+\frac{\|\nabla_{S_h} v_h\|_{0,S_h}}{\|G_hv_h\|_{0,S_h}}-1\right)\\
&\geq \|G_hv_h\|_{0,S_h}.	
 	\end{split}
\end{equation}

Thus, we obtain \eqref{inq:norm} by combining \eqref{inq:norm14} and \eqref{inq:norm15}. Moreover, $\norm{\cdot}$ being a norm on $V_h$ is a direct consequence of \eqref{inq:norm}. The proof is completed.
\end{proof}

We are now in a favorable position to establish the coercivity of the bilinear form $a_h(\cdot, \cdot)$, as demonstrated in the subsequent lemma.

\begin{lemma}\label{lem:wellposedness}
  For any $v_h\in V_h$, there exists a positive constant $\varrho$ such that
  \begin{equation}\label{ineq:wellposedness}
     a_h(v_h,v_h)\geq \varrho\norm{v_h}^2,
  \end{equation}
  provided the penalty parameter $\gamma$ is chosen to be sufficiently large.
  Furthermore, problem \eqref{equ:weak} possesses a unique solution.
\end{lemma}
\begin{proof}
For the sake of convenience, in the proof of this lemma, we will use $C$ with or without subscripts to denote generic positive constants.   By referring to the definition of $a_h(\cdot,\cdot)$, we deduce:
  \begin{equation}\label{ineq:cov1}
  	\begin{split}
  		a_h(v_h,v_h)=&\|\div_{S_h}G_hv_h\|_{0,S_h}^2-2\sum_{E\in \mathcal E_h}(\avg{\div_{S_h}G_hv_h},\jp{G_hv_h\cdot\mathbf n_E})_E\\
    &+\frac\gamma h\sum_{E\in \mathcal E_h}\|\jp{G_hv_h\cdot\mathbf n_E}\|_{0,E}^2+\gamma_{\text{stab}}\mathfrak S_h(v_h,v_h),
  	\end{split}
  \end{equation}
 with particular emphasis on the second term on the right-hand side. The trace inequality and the inverse inequality imply that: 
  \begin{equation}\label{ineq:cov2}
  	\begin{split}
\sum_{E\in\mathcal E_h}h\|\avg{\div_{S_h}G_hv_h}\|_{0,E}^2&\leq C\sum_{\tau\in\mathcal T_h}h\|\div_{S_h}G_hv_h\|_{0,\partial\tau}^2\\
&\leq C_1\sum_{\tau\in \mathcal T_h} \|\div_{S_h}G_hv_h\|_{0,\tau}^2=C_1\|\div_{S_h}G_hv_h\|_{0,S_h}^2.
  	\end{split}
  \end{equation}
 By directly applying the Cauchy-Schwarz inequality, Young's inequality with parameter $\epsilon$, and \eqref{ineq:cov2}, we obtain:
  \begin{equation}\label{ineq:cov3}
  	\begin{split}
&-\sum_{E\in \mathcal E_h}(\avg{\div_{S_h}G_hv_h},\jp{G_hv_h\cdot\mathbf n_E})_E\\
\leq &\sum_{E\in \mathcal E_h}h^{\frac12}\|\avg{\div_{S_h}G_hv_h}\|_{0,E}h^{-\frac12}\|\jp{G_hv_h\cdot\mathbf n_E}\|_{0,E}\\
\leq &\Big(\sum_{E\in \mathcal E_h}h\|\avg{\div_{S_h}G_hv_h}\|_{0,E}^2\Big)^{\frac12}\Big(\sum_{E\in \mathcal E_h}
h^{-1}\|\jp{G_hv_h\cdot\mathbf n_E}\|_{0,E}^2\Big)^{\frac12}\\
\leq &\epsilon\sum_{E\in \mathcal E_h}h^{\frac12}\|\avg{\div_{S_h}G_hv_h}\|_{0,E}^2+\frac1\epsilon
\sum_{E\in \mathcal E_h}h^{-1}\|\jp{G_hv_h\cdot\mathbf n_E}\|_{0,E}^2\\
\leq &C_1\epsilon \|\div_{S_h}G_hv_h\|_{0,S_h}^2+\frac1\epsilon
\sum_{E\in \mathcal E_h}h^{-1}\|\jp{G_hv_h\cdot\mathbf n_E}\|_{0,E}^2.
  	\end{split}
  \end{equation}
Combining \eqref{ineq:cov1}--\eqref{ineq:cov3}, we find:
  \begin{equation}\label{ineq:cov4}
  	\begin{split}
a_h(v_h,v_h)\geq&\, \frac12\|\div_{S_h}G_hv_h\|_{0,S_h}^2+\left(\frac12-2C_1\epsilon\right)\|\div_{S_h}G_hv_h\|_{0,S_h}^2\\
&+\left(\lambda-\frac2\epsilon\right)\sum_{E\in\mathcal E_h}h^{-1}\|\jp{G_hv_h\cdot\mathbf n_E}\|_{0,E}^2+\gamma_{\text{stab}}\mathfrak S_h(v_h,v_h).
  	\end{split}
  \end{equation}
 Selecting $\epsilon=\frac{1}{4C_1}$ and $\lambda>\frac2\epsilon$, we deduce the conclusion \eqref{ineq:wellposedness} by employing \eqref{ineq:cov4} and \eqref{ineq:cov2} with
 \[\varrho=\min\left\{\frac{1}{2(1+C_1)},\frac{\lambda-8C_1}{1+C_1},\gamma_{\text{stab}}\right\}.\]
  Consequently, in accordance with the Lax-Milgram Lemma, we conclude that the
discrete formulation \eqref{equ:weak} possesses a unique solution.
\end{proof}

\section{Error estimate}
\label{sec:err}
In this section, we will conduct an \textit{a priori} error analysis for the proposed  recovery based linear surface finite element method.

\subsection{Interpolation}
We introduce a Cl\'ement type interpolation operator $\mathcal I_h: L^2(S_h) \to \mathcal V_h$ defined as
\begin{equation}\label{def:interpolation}
  \mathcal{I}_h v = \sum_{P \in \mathcal{N}_h} \mathcal{I}_h v(P) \phi_P,
\end{equation}
where $\mathcal I_h v(P) = \frac{1}{|\Omega_P|} \int_{\Omega_P} v \dif\sigma_h$.

We can establish the following approximation property for the Cl\'ement type interpolation operator:
\begin{lemma}\label{lem:propertyofinterpolation}
Let $v \in H_h^1(S_h)$. Then, we have the following inequality:
\begin{equation}\label{ineq:propertyofinterpolation}
  \|v - \mathcal{I}_h v\|_{0, S_h} \lesssim h \left(|v|_{1, S_h}^2 + \sum_{E \in \mathcal{E}_h} h^{-1} \|\jp{v}\|_{0, E}^2\right)^{\frac{1}{2}}.
\end{equation}
\end{lemma}

\begin{proof}
For any $\tau \in \mathcal{T}_h$, let $\omega_\tau=\cup_{\bar\tau'\cap\bar\tau\neq\varnothing}\tau'$
denote the union of elements in $\mathcal{T}_h$ neighboring $\tau$. Furthermore, we define $\bar{v} = \frac{1}{|\omega_\tau|} \int_{\omega_\tau} v \, \mathrm{d}\sigma_h$. By applying the triangle inequality, the Cauchy-Schwarz inequality and the definition of $\mathcal I_h$, we obtain:
\begin{equation}\label{ineq:pro1}
	\begin{split}
		\|v - \mathcal{I}_h v\|_{0, \tau}^2 &\lesssim \|v - \bar{v}\|_{0, \tau}^2 + \|\mathcal{I}_h(v - \bar{v})\|_{0, \tau}^2 \\
  &\lesssim \|v - \bar{v}\|_{0, \omega_\tau}^2 + \bigg\|\sum_{P \in \bar{\tau}} \mathcal{I}_h(v - \bar{v})(P)\phi_P\bigg\|_{0, \tau}^2 \\
  &\lesssim \|v - \bar{v}\|_{0, \omega_\tau}^2 + h^2\sum_{P \in \bar{\tau}} \left|\frac{1}{|\Omega_P|} \int_{\Omega_P} (v - \bar{v}) \dif
  \sigma_h\right|^2 \\
  &\lesssim \|v - \bar{v}\|_{0, \omega_\tau}^2.
	\end{split}
\end{equation}
Summing over $\tau \in \mathcal{T}_h$, we arrive at:
\begin{equation}\label{ineq:pro2}
  \|v - \mathcal{I}_h v\|_{0, S_h}^2 \lesssim \sum_{\tau \in \mathcal{T}_h} \|v - \bar{v}\|_{0, \omega_\tau}^2 \lesssim h^2\left(|v|_{1,S_h}^2 + \sum_{E \in \mathcal{E}_h} h^{-1} \|\jp{v}\|_{0, E}^2\right).
\end{equation}
In the last inequality, we have employed the Poincaré--Friedrichs inequality for piecewise $H^1$ functions (see \cite{brenner2003}). Thus, the desired result follows.
\end{proof}

\subsection{Weak approximation of gradient recovery operator}

In general, the direct bound for $\|G_hv_h-\nabla_{S_h}v_h\|_{0,S_h}$ is not available. However, in the following lemma, we present a weak approximation of $G_hv_h-\nabla_{S_h}v_h$, which plays a crucial role in the a priori error estimate.

\begin{lemma}\label{lem:weakapx}
For any $v\in H^2_h(S_h)$ and $v_h\in V_h$, the following inequality holds:
\begin{equation}\label{ineq:weakapx}
    \int_{S_h} \nabla_{S_h} v\cdot(G_hv_h-\nabla_{S_h} v_h)\dif \sigma_h\lesssim
     h\Big(|v|_{2,S_h}^2+\sum_{E\in\mathcal E_h}h^{-1}\|\jp{\nabla_{S_h}v}\|_{0,E}^2
     \Big)^{\frac12}\norm{v_h}.
\end{equation}
\end{lemma}

\begin{proof}
Let $\bm w_h=\mathcal I_h(\nabla_{S_h}v)$. We first establish the following equality:
\begin{equation}\label{ineq:weakapx1}
    \int_{S_h}\sum_{P\in\mathcal N_h}\bm w_h(P)\cdot G_hv_h(P)\phi_P\dif\sigma_h=\int_{S_h} \bm w_h\cdot\nabla_{S_h} v_h\dif\sigma_h.
\end{equation}
To prove this, note that $\mathrm{supp}(\phi_P)=\bar\Omega_P$ and
$\nabla_{S_h}v_h$ is a constant vector on each facet triangle $\tau$, we derive that:
\begin{equation}\label{eq:eq1}
    \begin{split}
    &\,\,\int_{S_h}\sum_{P\in\mathcal N_h}\bm w_h(P)\cdot G_hv_h(P)\phi_P\dif\sigma_h
    =\sum_{P\in\mathcal N_h}\bm w_h(P)\cdot G_hv_h(P)\int_{\Omega_P}\phi_P\dif\sigma_h \\
    =&\,\,\frac13\sum_{P\in\mathcal N_h}\bm w_h(P)\cdot\int_{\Omega_P}\nabla_{S_h}v_h\dif\sigma_h
    =\frac13\sum_{P\in\mathcal N_h}\bm w_h(P)\cdot\sum_{\tau\subset \Omega_P}\int_{\tau}\nabla_{S_h}v_h\dif\sigma_h,
    \end{split}
\end{equation}
here in the second equality, we have used the fact that:
$\int_{\Omega_P}\phi_P\dif\sigma_h=\frac{|\Omega_P|}{3},$
which can be easily checked by the direct calculation. Change the order of summation yields:
\begin{equation}\label{eq:eq2}
    \begin{split}
&\,\,\frac13\sum_{P\in\mathcal N_h}\bm w_h(P)\cdot\int_{\Omega_P}\nabla_{S_h}v_h\dif\sigma_h = \frac13\sum_{\tau\in\mathcal T_h}\int_\tau \left(\sum_{P\in\mathcal N_h\cap\bar\tau}\bm w_h(P)\right)\cdot\nabla_{S_h}v_h\dif\sigma_h\\
=& \,\, \sum_{\tau\in\mathcal T_h}\sum_{P\in\mathcal N_h\cap\bar\tau}\bm w_h(P)\cdot\int_{\tau}\nabla_{S_h}v_h \phi_P\dif\sigma_h
=\sum_{\tau\in\mathcal T_h}\int_\tau\sum_{P\in\mathcal N_h\cap\bar\tau}(\bm w_h(P)\phi_P)\cdot\nabla_{S_h}v_h\dif\sigma_h\\
    =&\,\,\sum_{\tau\in\mathcal T_h}\int_\tau\sum_{P\in\mathcal N_h}(\bm w_h(P)\phi_P)\cdot\nabla_{S_h}v_h\dif\sigma_h=\int_{S_h} \bm w_h\cdot\nabla_{S_h} v_h\dif\sigma_h.
\end{split}
\end{equation}
Equations \eqref{eq:eq1} and \eqref{eq:eq2} give \eqref{ineq:weakapx1}.

Next, for $v\in H_h^2(S_h)$, we use \eqref{ineq:weakapx1} and the fact that
 \[\int_{S_h}\nabla_{S_h} v\cdot G_hv_h\dif\sigma_h = \sum_{P\in\mathcal N_h}\int_{\Omega_P}
\nabla_{S_h} v\cdot G_hv_h(P)\phi_P\dif\sigma_h\]
to deduce that:
\begin{equation}\label{ineq:weakapx2}
    \begin{split}
   &\int_{S_h} \nabla_{S_h} v\cdot(G_hv_h-\nabla_{S_h} v_h)\dif\sigma_h\\
   =&\sum_{P\in\mathcal N_h}\int_{\Omega_P}
   (\nabla_{S_h} v-\mathcal I_h(\nabla_{S_h} v)(P))\cdot G_hv_h(P)\phi_P\dif\sigma_h\\
   &\quad-\int_{S_h}(\nabla_{S_h} v-\mathcal I_h(\nabla_{S_h} v))\cdot \nabla_{S_h} v_h\dif\sigma_h.
    \end{split}
\end{equation}

Using the Cauchy-Schwarz inequality, Lemma \ref{lem:propertyofinterpolation}, Lemma \ref{lem:norm}, and piecewise  Poincar\'e–Friedrichs inequality, we estimate the second term on the right-hand side of \eqref{ineq:weakapx2} as follows:
\begin{equation}\label{ineq:weakapx3}
    \begin{split}
  \Big|\int_{S_h}(\nabla_{S_h} v-\mathcal I_h(\nabla_{S_h} v))\cdot \nabla_{S_h} v_h\dif\sigma_h\Big|
  &\lesssim\|\nabla_{S_h} v-\mathcal I_h(\nabla_{S_h} v)\|_{0,S_h}|v_h|_{1,S_h}\\
  &\lesssim h\Big(|v|_{2,S_h}^2+\sum_{E\in\mathcal E_h}h^{-1}\|\jp{\nabla_{S_h}v}\|_{0,E}^2
     \Big)^{\frac12}\norm{v_h}.
    \end{split}
\end{equation}

To estimate the first term on the right-hand side of \eqref{ineq:weakapx2}, we use the Cauchy-Schwarz inequality, the norm equivalence on $V_h$, the piecewise Poincar\'e-Friedrichs inequality \cite{brenner2003}, and Lemma \ref{lem:norm}:
\begin{equation}\label{ineq:weakapx4}
    \begin{aligned}
      &\,\,\Big|\sum_{P\in\mathcal N_h}\int_{\Omega_P}
   (\nabla_{S_h} v-\mathcal I_h(\nabla_{S_h} v)(P))\cdot G_hv_h(P)\phi_P\dif\sigma_h\Big| \\
 \lesssim &\,\,\sum_{P\in\mathcal N_h}\|\nabla_{S_h} v-\mathcal I_h(\nabla_{S_h} v)(P)\|_{0,\Omega_P}
 \|G_hv_h(P)\phi_P\|_{0,\Omega_P}\\
 =&\,\,\sum_{P\in\mathcal N_h}\Big\|\nabla_{S_h} v-\frac{1}{|\Omega_P|}\int_{\Omega_P}
 \nabla_{S_h} v\dif\sigma_h\Big\|_{0,\Omega_P}|G_hv_h(P)|\|\phi_P\|_{0,\Omega_P}\\
 \lesssim&\,\,\Big(\sum_{P\in\mathcal N_h}\Big\|\nabla_{S_h} v-\frac{1}{|\Omega_P|}\int_{\Omega_P}
 \nabla_{S_h} v\dif\sigma_h\Big\|_{0,\Omega_P}^2\Big)^{\frac12}\Big(\sum_{P\in\mathcal N_h}|G_hv_h(P)|^2\|\phi_P\|_{0,\Omega_P}^2\Big)^{\frac12}\\
 \lesssim&\,\, h\Big(|v|_{2,S_h}^2+\sum_{E\in\mathcal E_h}h^{-1}\|\jp{\nabla_{S_h}v}\|_{0,E}^2\Big)^{\frac12}
 \|G_hv_h\|_{0,S_h}\\
 \lesssim&\,\, h\Big(|v|_{2,S_h}^2+\sum_{E\in\mathcal E_h}h^{-1}\|\jp{\nabla_{S_h}v}\|_{0,E}^2\Big)^{\frac12}
 \norm{v_h}.
    \end{aligned}
\end{equation}
By combining \eqref{ineq:weakapx3} and \eqref{ineq:weakapx4} with \eqref{ineq:weakapx2}, we obtain the desired estimate.
\end{proof}

\subsection{Approximation property of gradient recovery operator} In this subsection, we shall establish more  approximation properties of the gradient recovery operator $G_h$.

To start with, we proceed to establish an error estimate involving the difference between the continuous differential operator and the discrete differential operator. The estimate techniques are standard, for instance, see \cite{dziuk1988finite}.

\begin{lemma}\label{lem:pro}
Let $u\in H^3(S)$. Then, the ensuing estimate is valid:
\begin{equation}\label{lem:pro1}
\|\nabla_{S_h}(\Delta_Su)^\ele-(\nabla_S\Delta_Su)^\ele\|_{0,S_h}\lesssim h\|u\|_{3,S}.
\end{equation}
Likewise, if $u\in H^2(S)$, we have:
\begin{equation}\label{lem:pro2}
\|(\Delta_Su)^\ele-\div_{S_h}(\nabla_Su)^\ele\|_{0,S_h}\lesssim h\|u\|_{2,S}.
\end{equation}
\end{lemma}

\begin{proof}
To establish \eqref{lem:pro1}, we begin by utilizing \eqref{equ:transform} to obtain the expression:
\begin{equation}\label{lem:pro11}
\nabla_{S_h}(\Delta_Su)^\ele(x)=(\mathbf P_h(\mathbf I-d\mathbf H)\mathbf P)(x)\nabla_S\Delta_Su(p(x)), \quad \text{for } x\in\tau.
\end{equation}
Next, considering that $\nabla_S\Delta_Su$ is tangential, for $x\in\tau$, we have, using $\mathbf n(x)=\mathbf n(p(x))$:
\begin{equation}\label{lem:pro12}
(\nabla_S\Delta_Su)^\ele(x)=\nabla_S\Delta_Su(p(x))=\mathbf P(p(x))\nabla_S\Delta_Su(p(x))=\mathbf P(x)\nabla_S\Delta_Su(p(x)).
\end{equation}
Applying the property $\mathbf P^2=\mathbf P$, we derive the relation:
\begin{equation}\label{lem:pro13}
\mathbf P_h(\mathbf I-d\mathbf H)\mathbf P(x)-\mathbf P(x)=(\mathbf P_h-\mathbf P+d\mathbf P_h\mathbf H)\mathbf P(x).
\end{equation}
By utilizing \eqref{geo1} and \eqref{geo3}, we obtain the bound:
\begin{equation}\label{lem:pro14}
 \| (\mathbf P_h-\mathbf P+d\mathbf P_h\mathbf H)\mathbf P(x)\|_{L^\infty(S_h)}\lesssim \|\mathbf P_h-\mathbf P\|_{L^\infty(S_h)}+\|d\|_{L^\infty(S_h)}\lesssim h.
\end{equation}
Through a change of integration domain, \eqref{lem:pro1} follows from \eqref{geo5}, \eqref{lem:pro11}, \eqref{lem:pro12}, and \eqref{lem:pro14}.

For \eqref{lem:pro2}, once again employing \eqref{equ:transform}, we obtain, for $x\in\tau$:
\begin{equation}\label{equ:div1}
\div_{S_h}(\nabla_Su)^\ele(x)=\mathrm{tr}\Big((\mathbf P_h(\mathbf I-d\mathbf H)\mathbf P)(x)\nabla_S^2u(p(x))\Big),
\end{equation}
and
\begin{equation}\label{equ:div2}
(\Delta_Su)^\ele(x)=\mathrm{tr}\Big(\mathbf P^2(x)\nabla_S^2u(p(x))\Big).
\end{equation}
Following a similar argument as in the proof of \eqref{lem:pro1}, we readily establish \eqref{lem:pro2}, completing the proof.
\end{proof}

\begin{remark}
In the proof of \eqref{lem:pro1}, we do not use the special property of $\Delta_Su$. Actually, we have a more general estimate compared to \eqref{lem:pro1}:
\begin{equation}\label{inq:pro11}
\|\nabla_{S_h}u^\ele-(\nabla_Su)^\ele\|_{0,S_h}\lesssim h\|u\|_{1,S} \qquad\forall u\in H^1(S).
\end{equation}
\end{remark}

Building upon the aforementioned lemma (Lemma \ref{lem:pro}), we can further establish additional approximation properties for the gradient recovery operator.
\begin{lemma}\label{lem:app}
Let $u\in H^3(S)$, then the following inequalities hold:
\begin{align}
\label{inq:app1} |(\nabla_Su)^\ele-G_hu_I|_{1,S_h}&\lesssim h\|u\|_{3,S}, \\
\label{inq:app2} \sum_{E\in\mathcal E_h}h\|\avg{(\Delta_Su)^\ele-\div_{S_h}G_hu_I}\|_{0,E}^2&\lesssim h^2\|u\|_{3,S}^2, \\
\label{inq:app3} \sum_{E\in\mathcal E_h}h^{-1}\|\jp{G_hu_I\cdot\mathbf n_E}\|_{0,E}^2&\lesssim h^2\|u\|_{3,S}^2.
\end{align}
\end{lemma}

\begin{proof}
We start with the estimation of the inequality \eqref{inq:app1}. Employing the triangle inequality, the inverse inequality, and the interpolation error estimate, we obtain:
\begin{equation}\label{inq:app11}
	\begin{split}
|(\nabla_Su)^\ele-G_hu_I|_{1,S_h} &\leq |(\nabla_Su)^\ele-I_h(\nabla_Su)^\ele|_{1,S_h}+|I_h(\nabla_Su)^\ele-G_hu_I|_{1,S_h}\\
 &\lesssim h\Big(\sum_{\tau\in\mathcal T_h}|(\nabla_Su)^\ele|_{2,\tau}^2\Big)^{\frac12}+h^{-1}\|I_h(\nabla_Su)^\ele-G_hu_I\|_{0,S_h}.		
	\end{split}
\end{equation}
Adding and subtracting $(\nabla_Su)^\ele$ in the second term, we have:
\begin{equation}\label{inq:app12}
	\begin{split}
&h^{-1}\|I_h(\nabla_Su)^\ele-G_hu_I\|_{0,S_h}\\
\lesssim& h^{-1}(\|I_h(\nabla_Su)^\ele-(\nabla_Su)^\ele\|_{0,S_h}+\|(\nabla_Su)^\ele-G_hu_I\|_{0,S_h})\\
\lesssim& h\Big(\sum_{\tau\in\mathcal T_h}|(\nabla_Su)^\ele|_{2,\tau}^2\Big)^{\frac12}+h^{-1}\|\nabla_Su-(G_hu_I)^\ell\|_{0,S}.
	\end{split}
\end{equation}
Then, utilizing the consistency \eqref{equ:consistency}, Lemma \ref{lem:equivalence} and the inverse inequality we get:
\begin{equation}\label{inq:app13}
	\begin{split}
		  |(\nabla_Su)^\ele-G_hu_I|_{1,S_h}& \lesssim h\Big(\sum_{\tau\in\mathcal T_h}|(\nabla_Su)^\ele|_{2,\tau}^2\Big)^{\frac12}+h^{-1}\|\nabla_Su-(G_hu_I)^\ell\|_{0,S}\\
 &\lesssim h\|u\|_{3,S}.
	\end{split}
\end{equation}
Then, we consider \eqref{inq:app2}. Using the trace inequality, we obtain:
\begin{equation}\label{inq:app21}
	\begin{split}
		&\sum_{E\in\mathcal E_h}h\|\avg{(\Delta_Su)^\ele-\div_{S_h}G_hu_I}\|_{0,E}^2\\
\lesssim &\sum_{\tau\in\mathcal T_h}h\|(\Delta_Su)^\ele-\div_{S_h}G_hu_I\|_{0,\partial\tau}^2\\
 \lesssim &\|(\Delta_Su)^\ele-\div_{S_h}G_hu_I\|_{0,S_h}^2+h^2|(\Delta_Su)^\ele|_{1,S_h}^2.
	\end{split}
\end{equation}
An application of \eqref{inq:app1} and \eqref{lem:pro2} leads to:
\begin{equation}\label{inq:app21a}
	\begin{split}
&\|(\Delta_Su)^\ele-\div_{S_h}G_hu_I\|_{0,S_h}^2\\
&\lesssim \|(\Delta_Su)^\ele-\div_{S_h}(\nabla_S u)^\ele\|_{0,S_h}^2
+\|\div_{S_h}(\nabla_S u)^\ele-\div_{S_h}G_hu_I\|_{0,S_h}^2\\
&\lesssim h^2\|u\|_{3,S}^2.
	\end{split}
\end{equation}
Combining \eqref{inq:app21a} with \eqref{inq:app21}, using Lemma \ref{lem:equivalence}, we obtain \eqref{inq:app2}.

Now, we proceed to the proof of inequality \eqref{inq:app3}. For any edge $E\in\mathcal E_h$, it follows from the triangle inequality that:
\begin{equation}\label{inq:app31}
h^{-1}\|\jp{G_hu_I\cdot\mathbf n_E}\|_{0,E}^2
\lesssim h^{-1}\Big(\|\jp{(G_hu_I-(\nabla_Su)^\ele)\cdot\mathbf n_E}\|_{0,E}^2
+\|\jp{(\nabla_Su)^\ele\cdot\mathbf n_E}\|_{0,E}^2\Big).
\end{equation}
By following similar arguments as in \eqref{inq:app21}, we get:
\begin{equation}\label{inq:app32}
	\begin{split}
		& \sum_{E\in\mathcal E_h}h^{-1}\|\jp{(G_hu_I-(\nabla_Su)^\ele)\cdot\mathbf n_E}\|_{0,E}^2\\
\lesssim& \sum_{\tau\in\mathcal T_h}h^{-1}\|(G_hu_I-(\nabla_Su)^\ele)\cdot\mathbf n_E\|_{0,\partial\tau}^2\\
\lesssim& \sum_{\tau\in\mathcal T_h}(h^{-2}\|G_hu_I-(\nabla_Su)^\ele\|_{0,\tau}^2
+|G_hu_I-(\nabla_Su)^\ele|_{1,\tau}^2)\\
\lesssim& h^2\|u\|_{3,S}^2.
	\end{split}
\end{equation}
For the second term of \eqref{inq:app31}, we derive by Lemma \ref{lem:equivalenceedge} and \eqref{geo6} that:
\begin{equation*}
	\begin{split}
    \|\jp{(\nabla_Su)^\ele\cdot\mathbf n_E}\|_{0,E}^2&\lesssim \|\jp{\nabla_Su\cdot\mathbf n_E}\|_{0,E^\ell}^2
    =\int_{E^\ell}\Big(\nabla_Su\cdot\mathbf n_E^+-\nabla_Su\cdot\mathbf n_E^-\Big)^2\dif s\\
    &\lesssim \int_{E^\ell}\Big(\nabla_Su\cdot(\mathbf P\mathbf n_E^+-\mathbf n_{E^\ell}^+)\Big)^2\dif s
+\int_{E^\ell}\Big(\nabla_Su\cdot(\mathbf P\mathbf n_E^--\mathbf n_{E^\ell}^-)\Big)^2\dif s\\
    &\lesssim h^4\|\nabla_Su\|_{0,E^\ell}^2.
	\end{split}
\end{equation*}
Summing over all $E\in\mathcal E_h$ and using the trace inequality, we have:
\begin{equation}\label{inq:app33}
  \sum_{E\in\mathcal E_h}h^{-1}\|\jp{(\nabla_Su)^\ele\cdot\mathbf n_E}\|_{0,E}^2\lesssim h^2\|u\|_{3,S}^2.
\end{equation}
From \eqref{inq:app31}, \eqref{inq:app32}, and \eqref{inq:app33}, we obtain \eqref{inq:app3}, and the proof is completed.
\end{proof}

\subsection{Energy error estimate}
In this subsection, we aim to establish an \textit{a priori} error estimate in the energy norm. To facilitate the energy error estimation, we initiate the discussion by introducing the subsequent lemma.

\begin{lemma}\label{lemma:T1}
Consider the linear functional $l(\cdot) $ defined in \eqref{equ:lfun}, and let $l_h(\cdot)$ denote its discrete counterpart as defined in \eqref{equ:discrete_lfun}. The discrepancy between these functionals satisfies the following inequality:
\begin{equation}\label{inq:termT1}
   |l_h(v_h)-l(v_h^\ell)| \lesssim h^2\|f\|_{0,S}\norm{v_h}.
\end{equation}
\end{lemma}
\begin{proof}
We provide a standard proof for the sake of completeness, elucidating specific details. Let $Q_0^{S_h}w=\frac{1}{|S_h|}\int_{S_h}w\dif\sigma_h$ denote the $L^2$ projection of function $w$ onto the space of
constant functions on $S_h$. Similarly, $Q_0^Sw=\frac{1}{|S|}\int_Sw\dif\sigma$ is defined accordingly. By utilizing $\int_Sf\dif\sigma=0$, we can establish the equivalence:
\[(f,v_h^\ell)_S=(f,v_h^\ell-Q_0^Sv_h^\ell)_S=(f-Q_0^{S_h} f^\ell,v_h^\ell-Q_0^Sv_h^\ell)_S.\]
Applying the definition of $f_h$ and employing \eqref{geo5}, the Poincar\'e inequality, and Lemma \ref{lem:equivalence}, we deduce that:
\begin{align*}
 l_h(v_h)-l(v_h^\ell)&=(f_h,v_h)_{S_h}-(f,v_h^\ell)_S\\
 &=(f^\ele-Q_0^{S_h} f^\ele,v_h)_{S_h}-(f-Q_0^{S_h} f^\ele,v_h^\ell-Q_0^Sv_h^\ell)_S\\
 &=(f^\ele-Q_0^{S_h} f^\ele,v_h-Q_0^Sv_h^\ell)_{S_h}-(\mu_h(f^\ele-Q_0^{S_h} f^\ele),v_h-Q_0^Sv_h^\ell )_{S_h}\\
 &=((1-\mu_h)(f^\ele-Q_0^{S_h} f^\ele),v_h-Q_0^Sv_h^\ell)_{S_h}\\
 &\leq\|1-\mu_h\|_{L^{\infty}(S_h)}\|f^\ele-Q_0^{S_h} f^\ele\|_{0,S_h}\|v_h-Q_0^Sv_h^\ell\|_{0,S_h}\\
 &\lesssim h^2\|f^\ele\|_{0,S_h}|v_h|_{1,S_h}\lesssim h^2\|f\|_{0,S}\norm{v_h}.
\end{align*}
In the last inequality, we have utilized \eqref{inq:norm} to bound the $H^1$ norm of $v_h$, thus establishing the desired result.
\end{proof}

To establish the weak estimate involving the recovered gradient of finite element functions and the exact solution, we introduce the following theorem.

\begin{lemma}\label{lemma:T3}
Let $u\in H^4(S)$ be the solution of \eqref{eq:weak} and $v_h \in S_h$. Then, the ensuing estimate holds:
\begin{equation}\label{inq:termT3}
(\nabla_{S_h}(\Delta_Su)^\ele,\nabla_{S_h}v_h-G_hv_h)_{S_h}\lesssim h\|u\|_{4,S}\norm{v_h}.
\end{equation}
\end{lemma}

\begin{proof}
According to Lemma \ref{lem:weakapx}, we deduce the following intermediate inequality:
\begin{equation}\label{T3:1}
	\begin{split}
&\,\,(\nabla_{S_h}(\Delta_Su)^\ele,\nabla_{S_h}v_h-G_hv_h)_{S_h} \\
\lesssim&\,\, h\left(|(\Delta_Su)^\ele|_{2,S_h}^2+
\sum_{E\in\mathcal E_h}h^{-1}\|\jp{\nabla_{S_h}(\Delta_Su)^\ele}\|_{0,E}^2\right)^{\frac12}\norm{v_h}.	
	\end{split}
\end{equation}
For the jump term of \eqref{T3:1}, by adding and subtracting $(\nabla_S\Delta_Su)^\ele$ and using the trace inequality, we obtain the following estimate:
\begin{equation}\label{T3:2}
	\begin{split}
&  \sum_{E\in\mathcal E_h}h^{-1}\|\jp{\nabla_{S_h}(\Delta_Su)^\ele}\|_{0,E}^2\\
=&\sum_{E\in\mathcal E_h}h^{-1}\|\jp{(\nabla_{S_h}(\Delta_Su)^\ele-(\nabla_S\Delta_Su)^\ele)}\|_{0,E}^2\\
\lesssim&\sum_{\tau\in\mathcal T_h}\left(h^{-2}\|\nabla_{S_h}(\Delta_Su)^\ele-(\nabla_S\Delta_Su)^\ele\|_{0,\tau}^2+
|\nabla_{S_h}(\Delta_Su)^\ele-(\nabla_S\Delta_Su)^\ele|_{1,\tau}^2\right).
	\end{split}
\end{equation}
By \eqref{lem:pro1}, it is evident that:
\begin{equation}\label{T3:3}
	\begin{split}
  \sum_{\tau\in\mathcal T_h}h^{-2}\|\nabla_{S_h}(\Delta_Su)^\ele-(\nabla_S\Delta_Su)^\ele\|_{0,\tau}^2\lesssim \|u\|_{3,S}^2.
	\end{split}
\end{equation}
We derive, by the triangle inequality and Lemma \ref{lem:equivalence}, the following bound:
\begin{equation}\label{T3:4}
	\begin{split}
&\sum_{\tau\in\mathcal T_h}|\nabla_{S_h}(\Delta_Su)^\ele-(\nabla_S\Delta_Su)^\ele|_{1,\tau}^2\\
\lesssim&\sum_{\tau\in\mathcal T_h}|\nabla_{S_h}(\Delta_Su)^\ele|_{1,\tau}^2+|(\nabla_S\Delta_Su)^\ele|_{1,\tau}^2
\lesssim \|u\|_{4,S}^2.
	\end{split}
\end{equation}
Combining \eqref{T3:3}--\eqref{T3:4}, we obtain:
\begin{equation}\label{T3:6}
  \sum_{E\in\mathcal E_h}h^{-1}\|\jp{\nabla_{S_h}(\Delta_Su)^\ele}\|_{0,E}^2\lesssim \|u\|_{4,S}^2,
\end{equation}
which completes the proof.
\end{proof}

Then, we estimate the error between the continuous and discrete bilinear form.

\begin{lemma}\label{lemma:T4}
Let $u$ be the solution of the weak formulation \eqref{eq:weak}, and let $v_h \in S_h$ be an element of the finite element space. Then, the following inequality holds:
\begin{equation}\label{inq:termT4-rephrased}
     |-(\nabla_{S_h}(\Delta_Su)^\ele,G_hv_h)_{S_h}-a_h(u_I,v_h)| \lesssim h\|u\|_{3,S}\norm{v_h}.
\end{equation}
\end{lemma}
\begin{proof}
Starting with the definition of $a_h(\cdot,\cdot)$ and applying Green's formula on each elements, we obtain the following expression:
\begin{equation}\label{inq:T41-rephrased}
  	\begin{split}
  		  &\,\,-(\nabla_{S_h}(\Delta_Su)^\ele,G_hv_h)_{S_h}-a_h(u_I,v_h) \\
  =&\,\,((\Delta_Su)^\ele,\div_{S_h}G_hv_h)_{S_h}-\sum_{E\in\mathcal E_h}(\avg{(\Delta_Su)^\ele},\jp{G_hv_h\cdot\mathbf n_E})_e-a_h(u_I,v_h)\\
   =&\,\,((\Delta_Su)^\ele-\div_{S_h}G_hu_I,\div_{S_h}G_hv_h)_{S_h}\\
   &\,\,-\sum_{E\in\mathcal E_h}(\avg{(\Delta_Su)^\ele-\div_{S_h}G_hu_I},\jp{G_hv_h\cdot\mathbf n_E})_E\\
    &\,\,+\sum_{E\in\mathcal E_h}(\avg{\div_{S_h}G_hv_h},\jp{G_hu_I\cdot\mathbf n_E})_E\\
    &\,\,-\sum_{E\in\mathcal E_h}\frac\gamma h(\jp{G_hu_I\cdot\mathbf n_E},\jp{G_hv_h\cdot\mathbf n_E})_E-\gamma_{\text{stab}}\mathfrak S_h(u_I,v_h)\\
:=&\,\,\sum_{i=1}^5\mathfrak T_i.
  	\end{split}
  \end{equation}

Next, we proceed to estimate each term $\mathfrak T_i, i=1,\cdots, 5$, individually. Utilizing the Cauchy-Schwarz inequality, we find that:
\[((\Delta_Su)^\ele-\div_{S_h}G_hu_I,\div_{S_h}G_hv_h)_{S_h}\lesssim \|(\Delta_Su)^\ele-\div_{S_h}G_hu_I\|_{0,S_h}\norm{v_h}.\]
By applying \eqref{lem:pro2} and \eqref{inq:app1}, we deduce that:
\begin{align*}
  &\|(\Delta_Su)^\ele-\div_{S_h}G_hu_I\|_{0,S_h}\\
  &\leq \|(\Delta_Su)^\ele-\div_{S_h}(\nabla_Su)^\ele\|_{0,S_h}+\|\div_{S_h}(\nabla_Su)^\ele-\div_{S_h}G_hu_I\|_{0,S_h}\\
  &\lesssim h\|u\|_{3,S}.
\end{align*}
Hence, $\mathfrak T_1$ can be estimated as follows:
\begin{equation}\label{inq:T411-rephrased}
 \mathfrak T_1\lesssim h\|u\|_{3,S}\norm{v_h}.
\end{equation}
For $\mathfrak T_2$, employing \eqref{inq:app2}, yields:
\begin{equation}\label{inq:T42-rephrased}
	\begin{split}
		  \mathfrak T_2 & \leq \left(\sum_{E\in\mathcal E_h}h\|\avg{(\Delta_Su)^\ele-\div_{S_h}G_hu_I}\|_{0,E}^2\right)^{\frac12}\norm{v_h}\\
  &\lesssim h\|u\|_{3,S}\norm{v_h}.
	\end{split}
\end{equation}
Analogously, employing \eqref{inq:app3}, we obtain the estimates:
\begin{align}
 \label{inq:T43-rephrased}&\mathfrak T_3   \lesssim h\|u\|_{3,S}\norm{v_h},\\
  \label{inq:T44-rephrased}&\mathfrak T_4   \lesssim h\|u\|_{3,S}\norm{v_h}.
\end{align}
Finally, we address the stabilization term $\mathfrak T_5$. Using \eqref{inq:pro11}, \eqref{equ:consistency}, and the interpolation error estimate, we find:
\begin{equation}\label{inq:T45-rephrased}
	\begin{split}
		\mathfrak T_5 &=-\gamma_{\text{stab}}\mathfrak S_h(u_I,v_h)\\
  &\leq \gamma_{\text{stab}}\|\nabla_{S_h}u_I-G_hu_I\|_{0,S_h}\|\nabla_{S_h}v_h-G_hv_h\|_{0,S_h}\\
  &\leq \gamma_{\text{stab}}\Big(\|\nabla_{S_h}u_I-\nabla_{S_h}u^\ele\|_{0,S_h}+\|\nabla_{S_h}u^\ele-(\nabla_Su)^\ele\|_{0,S_h}\\
  &\qquad+\|(\nabla_Su)^\ele-G_hu_I\|_{0,S_h}\Big)\norm{v_h}\\
  &\lesssim h\|u\|_{3,S}\norm{v_h}.
	\end{split}
\end{equation}
By substituting \eqref{inq:T411-rephrased} -- \eqref{inq:T45-rephrased} into \eqref{inq:T41-rephrased}, we arrive at the desired result.
\end{proof}

Now, we are ready to estimate the error between the surface finite element solution and the interpolation of the exact solution.

\begin{lemma}\label{lem:super}
Let $u\in H^4(S)$ be the solution of the weak formulation \eqref{eq:weak}, and let $u_h\in V_h$ be the solution of the discrete problem \eqref{equ:weak}. Then, the following error estimate holds:
\begin{equation}\label{inq:super}
\norm{u_h-u_I}\lesssim h\|f\|_{0,S}.
\end{equation}
\end{lemma}

\begin{proof}
According to Lemma \ref{lem:wellposedness}, the coercivity of $a_h$, we obtain:
\begin{equation}\label{equ:normdualform}
\norm{u_h-u_I}\lesssim\sup_{v_h\in V_h}\frac{a_h(u_h-u_I, v_h)}{\norm{v_h}}.
\end{equation}
We proceed by expressing $a_h(u_h-u_I,v_h)$ as follows:
\begin{equation}\label{eq:err1}
\begin{aligned}
a_h(u_h-u_I,v_h) &= a_h(u_h,v_h)-a_h(u_I,v_h)
= l_h(v_h)-a_h(u_I,v_h)\\
&= l_h(v_h)-l(v_h^\ell)+l(v_h^\ell)-a_h(u_I,v_h).
\end{aligned}
\end{equation}
For the second $l(v_h^\ell)$ in the above equality, we can use the continuous equation \eqref{eq:problem} and the transformation \eqref{equ:transformbigrad} to obtain:
\begin{equation}\label{eq:err2}
\begin{aligned}
l(v_h^\ell) &= (f,v_h^\ell)_S=(\Delta_S^2u,v_h^\ell)_S=-(\nabla_S(\Delta_Su),\nabla_Sv_h^\ell)_S\\
&=((\mathbf R_h-\mathbf I)\nabla_S(\Delta_Su),\nabla_Sv_h^\ell)_S-(\mathbf R_h\nabla_S(\Delta_Su),\nabla_Sv_h^\ell)_S\\
&=((\mathbf R_h-\mathbf I)\nabla_S(\Delta_Su),\nabla_Sv_h^\ell)_S-(\nabla_{S_h}(\Delta_Su)^\ele,\nabla_{S_h}v_h)_{S_h}.
\end{aligned}
\end{equation}
By adding and subtracting $G_hv_h$ in the second term of \eqref{eq:err2}, we obtain:
\begin{equation}\label{eq:err4}
\begin{aligned}
a_h(u_h-u_I,v_h) =\,& l_h(v_h)-l(v_h^\ell)+((\mathbf R_h-\mathbf I)\nabla_S(\Delta_Su),\nabla_Sv_h^\ell)_S\\
\,&-(\nabla_{S_h}(\Delta_Su)^\ele,\nabla_{S_h}v_h-G_hv_h)_{S_h}\\
\,&-(\nabla_{S_h}(\Delta_Su)^\ele,G_hv_h)_{S_h}-a_h(u_I,v_h)\\
=\,&T_1+T_2+T_3+T_4.
\end{aligned}
\end{equation}
Lemmas \ref{lemma:T1}, \ref{lemma:T3}, \ref{lemma:T4}, and the regularity result \eqref{inq:regularity} imply that:
\begin{equation}\label{equ:T134}
T_1+T_3+T_4 \lesssim h\|f\|_{0, S}\norm{v_h}.
\end{equation}
For $T_2$, we employ the Cauchy-Schwarz inequality, \eqref{geo4}, \eqref{inq:norm}, \eqref{inq:regularity}, and the norm equivalence in Lemma \ref{lem:equivalence} to get:
\begin{equation}\label{equ:T2}
\begin{aligned}
T_2 &= ((\mathbf R_h-\mathbf I)\nabla_S(\Delta_Su),\nabla_Sv_h^\ell)_S
=((\mathbf R_h-\mathbf I)\mathbf P\nabla_S(\Delta_Su),\nabla_Sv_h^\ell)_S\\
&\leq \|(\mathbf R_h- \mathbf I)\mathbf P\|_{L^\infty(S)}|u|_{3,S}|v_h^\ell|_{1,S}
\lesssim h^2\|u\|_{3,S}\norm{v_h}\lesssim h^2\|f\|_{0, S}\norm{v_h}.
\end{aligned}
\end{equation}
Combining \eqref{equ:normdualform}, \eqref{eq:err4}, \eqref{equ:T134}, and \eqref{equ:T2} concludes our proof.
\end{proof}

We are now poised to present the main a priori error estimate:
\begin{theorem}\label{thm:prior}
 Let $u\in H^4(S)$ be the solution of \eqref{eq:weak}, and $u_h\in V_h$ be the solution of \eqref{equ:weak}. Then,
   \begin{equation}\label{1314}
     \|\Delta_Su-(\div_{S_h}G_hu_h)^\ele\|_{0,S}\lesssim h\|f\|_{0,S}.
   \end{equation}
\end{theorem}
\begin{proof}
By the triangle inequality, \eqref{lem:pro2}, \eqref{inq:app1} and Lemma \ref{lem:super}, we derive:
\begin{equation}
	\begin{split}
		  &\,\,\|\Delta_Su-(\div_{S_h}G_hu_h)^\ele\|_{0,S}\lesssim \|(\Delta_Su)^\ele-\div_{S_h}G_hu_h\|_{0,S_h}\\
 \lesssim &\,\, \|(\Delta_Su)^\ele-\div_{S_h}(\nabla_Su)^\ele\|_{0,S_h}+\|\div_{S_h}(\nabla_Su)^\ele-\div_{S_h}G_hu_I\|_{0,S_h}\\
&\qquad+\|\div_{S_h}G_hu_I-\div_{S_h}G_hu_h\|_{0,S_h}\\
\lesssim &\,\,h\|u\|_{4,S}\lesssim h\|f\|_{0,S}.
	\end{split}
\end{equation}
  Then the desired result follows.
\end{proof}

\subsection{$L^2$ error estimate} This subsection is dedicated to deriving the $L^2$ error estimate using a duality argument. For this purpose, we introduce the dual problem: Find $z\in H^2(S)/\mathbb{R}$ such that
\begin{equation}\label{dual:c}
a(v,z) = (g,v)_S \quad \forall v \in H^2(S)/\mathbb{R},
\end{equation}
along with its corresponding discrete approximation problem: Find $z_h\in V_h$ such that
\begin{equation}\label{dual:d}
a_h(v_h,z_h) = (g_h,v_h)_{S_h} \quad \forall v_h \in V_h,
\end{equation}
where $g = u - u_h^\ell - Q_0^S(u - u_h^\ell)$ and $g_h = u^\ele - u_h - Q_0^{S_h}(u^\ele - u_h)$.

For our proof, we require the $\mathbf P_h\mathbf{n}$ Lemma, a weak estimate result in Sobolev space $[W_1^1(S_h)]^3$, which was established in \cite{larsson2017continuous}.

\begin{lemma}\label{lem:phn}
For $\chi\in [W_1^1(S_h)]^3$, the following inequality holds:
\begin{equation}\label{inq:phn}
\int_{S_h}(\mathbf P_h\mathbf{n})\cdot\chi\dif\sigma_h \lesssim h^2\|\chi\|_{W_1^1(S_h)}.
\end{equation}
\end{lemma}

As an immediate consequence of the above $\mathbf P_h\mathbf{n}$ Lemma, we obtain the following result:

\begin{lemma}\label{lem:tranh2}
  Let $\phi, \psi\in H^4(S)$ be functions defined on the surface $S$. Then, the following estimates hold:
  \begin{align}
    \label{inq:tranh2div}((\Delta_S\phi)^\ele-\div_{S_h}(\nabla_S\phi)^\ele,(\Delta_S\psi)^\ele)_{S_h}&\lesssim h^2\|\phi\|_{3,S}\|\psi\|_{3,S},\\
   \label{inq:tranh2e} \sum_{E\in\mathcal E_h}((\Delta_S\phi)^\ele,\jp{(\nabla_S\psi)^\ele\cdot\mathbf n_E})_E & \lesssim h^2\|\phi\|_{4,S}\|\psi\|_{4,S}.
  \end{align}
\end{lemma}

\begin{proof}
 The inequality \eqref{inq:tranh2e} can be found in Appendix (C.33) of \cite{larsson2017continuous}. Hence, we only need to establish the validity of \eqref{inq:tranh2div}. However, directly applying \eqref{lem:pro2} would only yield an $\mathcal{O}(h)$ bound. To overcome this limitation, we begin by examining the relation between of $\mathbf P$ and $\mathbf P_h$. Utilizing \eqref{equ:div1} and \eqref{equ:div2}, we have the identity:
 \begin{equation}\label{inq:tranh2div1}
   (\div_S\nabla_S\phi)^\ele-\div_{S_h}(\nabla_S\phi)^\ele=\mathrm{tr}\Big((\mathbf P(x)-\mathbf P_h(I-d\mathbf H)\mathbf P(x))(\nabla_S^2\phi)^\ele\Big).
 \end{equation}
 Considering the geometric property \eqref{geo1}, we find that $\|d\|_{L^\infty(S_h)}\lesssim h^2$, which allows us to focus on the dominant part:
 \begin{equation}\label{inq:tranh2div2}
 \begin{aligned}
 	&(\mathrm{tr}((\mathbf n_h\otimes\mathbf n_h)\mathbf P(\nabla_S^2\phi)^\ele), (\Delta_S\psi)^\ele)_{S_h} \\
 	&= (\mathbf n_h^t(\mathbf P(\nabla_S^2\phi)^\ele)^t\mathbf n_h, (\Delta_S\psi)^\ele)_{S_h} \\
 	&= \int_{S_h}(\mathbf P\mathbf n_h)\cdot ((\Delta_S\psi)^\ele(\nabla_S^2\phi)^\ele \mathbf n_h)\dif\sigma_h\\
 	&\leq\|\int_{S_h}(\mathbf P\mathbf n_h)\cdot ((\Delta_S\psi)^\ele(\nabla_S^2\phi)^\ele)\dif\sigma_h\|_{\ell^\infty},
 \end{aligned}
 \end{equation}
 where the last norm $\|\cdot\|_{\ell^\infty}$ represents the $\ell^\infty$ norm of a vector.

To proceed, we leverage the identity:
\[\mathbf P\mathbf n_h=(1-\mathbf n\cdot\mathbf n_h)(\mathbf n+\mathbf n_h)-\mathbf P_h\mathbf n,\]
to rephrase the right-hand term of \eqref{inq:tranh2div2} as
 \begin{equation}\label{inq:tranh2div3}
 \begin{aligned}
  &\|\int_{S_h}(\mathbf P\mathbf n_h)\cdot ((\Delta_S\psi)^\ele(\nabla_S^2\phi)^\ele)\dif\sigma_h\|_{\ell^\infty}\\
  &=\|\int_{S_h}(1-\mathbf n\cdot\mathbf n_h)(\mathbf n+\mathbf n_h)\cdot ((\Delta_S\psi)^\ele(\nabla_S^2\phi)^\ele)-
  (\mathbf P_h\mathbf n)\cdot ((\Delta_S\psi)^\ele(\nabla_S^2\phi)^\ele)\dif\sigma_h\|_{\ell^\infty}\\
  &\lesssim h^2\|\phi\|_{2,S}\|\psi\|_{2,S}+\|\int_{S_h}(\mathbf P_h\mathbf n)\cdot ((\Delta_S\psi)^\ele(\nabla_S^2\phi)^\ele)\dif\sigma_h\|_{\ell^\infty},
 \end{aligned}
 \end{equation}
where, in the last inequality, we have used \eqref{geo7} and the Cauchy-Schwarz inequality. By employing Lemma \ref{lem:phn} and the Cauchy-Schwarz inequality once again, we arrive at:
 \begin{equation}\label{inq:tranh2div4}
 \begin{aligned}
  \|\int_{S_h}(\mathbf P_h\mathbf n)\cdot ((\Delta_S\psi)^\ele(\nabla_S^2\phi)^\ele)\dif\sigma_h\|_{\ell^\infty}
  &\lesssim h^2\|(\Delta_S\psi)^\ele(\nabla_S^2\phi)^\ele\|_{W_1^1(S_h)}\\
  &\lesssim h^2\|\phi\|_{3,S}\|\psi\|_{3,S},
 \end{aligned}
 \end{equation}
which completes the proof.
\end{proof}

To derive the $L^2$ error estimate, we  need to estimate the consistency error of the exact operator and the lift operator.

\begin{lemma} \label{lem:l2aa}
  Let $\phi, \psi\in H^3(S)$ be functions defined on the surface $S$. Then, the following inequality holds:
  \begin{equation}\label{inq:l2aa1}
    |a(\phi,\psi)-(\div_{S_h}(\nabla_S\phi)^\ele, \div_{S_h}(\nabla_S\psi)^\ele)_{S_h}|\lesssim h^2\|\phi\|_{3,S}\|\psi\|_{3,S}.
  \end{equation}
\end{lemma}

\begin{proof}
First, we use \eqref{geo5} to obtain:
\begin{equation}\label{inq:l2aa3}
	\begin{aligned}
&a(\phi,\psi)-((\Delta_S\phi)^\ele, (\Delta_S\psi)^\ele)_{S_h}=((\mu_h-1)(\Delta_S\phi)^\ele, (\Delta_S\psi)^\ele)_{S_h}\\
&\leq\|\mu_h-1\|_{L^\infty(S_h)}\|(\Delta_S\phi)^\ele\|_{0,S_h}\|(\Delta_S\psi)^\ele\|_{0,S_h}\lesssim h^2\|\phi\|_{2,S}\|\psi\|_{2,S}.
  \end{aligned}
\end{equation}
Next, we derive from \eqref{inq:tranh2div}, \eqref{lem:pro2}, and Lemma \ref{lem:equivalence} that
 \begin{equation}\label{inq:l2aa2}
    \begin{aligned}
   & ((\Delta_S\phi)^\ele, (\Delta_S\psi)^\ele)_{S_h}-(\div_{S_h}(\nabla_S\phi)^\ele, \div_{S_h}(\nabla_S\psi)^\ele)_{S_h}\\
   &=((\Delta_S\phi)^\ele-\div_{S_h}(\nabla_S\phi)^\ele, (\Delta_S\psi)^\ele)_{S_h}+
   ((\Delta_S\phi)^\ele, (\Delta_S\psi)^\ele-\div_{S_h}(\nabla_S\psi)^\ele)_{S_h}\\
   &-((\Delta_S\phi)^\ele-\div_{S_h}(\nabla_S\phi)^\ele, (\Delta_S\psi)^\ele-\div_{S_h}(\nabla_S\psi)^\ele)_{S_h}\\
   &\lesssim h^2\|\phi\|_{3,S}\|\psi\|_{3,S}.
  \end{aligned}
 \end{equation}
Combining \eqref{inq:l2aa3} and \eqref{inq:l2aa2}, we complete the proof.
 \end{proof}

By employing the previously established Lemma \ref{lem:l2aa}, we can derive the consistency error that arises between the continuous and discrete bilinear forms as follows:

\begin{lemma}
  \label{lem:l2a}
  Let $\phi, \psi\in H^4(S)$. Then, the following inequality holds:
  \begin{equation}\label{ineq:l2a}
    |a(\phi,\psi)-a_h(\phi_I,\psi_I)|\leq Ch^2\|\phi\|_{4,S}\|\psi\|_{4,S},
  \end{equation}
  where $\phi_I=I_h\phi$ and $\psi_I=I_h\psi$.
\end{lemma}

\begin{proof}
Using Lemma \ref{lem:l2aa}, we only need to bound the term:
\begin{equation}\label{inq:l2a1}
  (\div_{S_h}(\nabla_S\phi)^\ele, \div_{S_h}(\nabla_S\psi)^\ele)_{S_h} - a_h(\phi_I,\psi_I).
\end{equation}

First, we analyze the term of the integral on $S_h$ in $a_h(\phi_I,\psi_I)$. By adding and subtracting some terms, we obtain:
\begin{equation}\label{inq:l2a1m}
\begin{aligned}
&(\div_{S_h}(\nabla_S\phi)^\ele, \div_{S_h}(\nabla_S\psi)^\ele)_{S_h} - (\div_{S_h}G_h\phi_I,\div_{S_h}G_h\psi_I)_{S_h}\\
&=(\div_{S_h}((\nabla_S\phi)^\ele-G_h\phi_I), (\Delta_S\psi)^\ele)_{S_h}\\
&\quad+(\div_{S_h}((\nabla_S\phi)^\ele-G_h\phi_I), \div_{S_h}(\nabla_S\psi)^\ele-(\Delta_S\psi)^\ele)_{S_h}\\
&\quad+(\div_{S_h}((\nabla_S\psi)^\ele-G_h\psi_I), (\Delta_S\phi)^\ele)_{S_h}\\
&\quad+(\div_{S_h}((\nabla_S\psi)^\ele-G_h\psi_I), \div_{S_h}(\nabla_S\phi)^\ele-(\Delta_S\phi)^\ele)_{S_h}\\
&\quad-(\div_{S_h}((\nabla_S\phi)^\ele-G_h\phi_I), \div_{S_h}((\nabla_S\psi)^\ele-G_h\psi_I))_{S_h}.
\end{aligned}
\end{equation}

The second term, fourth term, and last term on the right-hand side of \eqref{inq:l2a1m} can be bounded using the Cauchy-Schwarz inequality, Lemma \ref{lem:pro2}, and \eqref{inq:app1} as follows:
\begin{align}
 & (\div_{S_h}((\nabla_S\phi)^\ele-G_h\phi_I), \div_{S_h}(\nabla_S\psi)^\ele-(\Delta_S\psi)^\ele)_{S_h}\lesssim h^2\|\phi\|_{3,S}\|\psi\|_{2,S},
 \label{inq:l2a2}\\
 & (\div_{S_h}((\nabla_S\psi)^\ele-G_h\psi_I), \div_{S_h}(\nabla_S\phi)^\ele-(\Delta_S\phi)^\ele)_{S_h} \lesssim h^2\|\phi\|_{2,S}\|\psi\|_{3,S},
 \label{inq:l2a3}\\
 & (\div_{S_h}((\nabla_S\phi)^\ele-G_h\phi_I), \div_{S_h}((\nabla_S\psi)^\ele-G_h\psi_I))_{S_h}\lesssim h^2\|\phi\|_{3,S}\|\psi\|_{3,S}.
 \label{inq:l2a4}
\end{align}
Substituting \eqref{inq:l2a2}, \eqref{inq:l2a3}, and \eqref{inq:l2a4} into \eqref{inq:l2a1m}, we obtain:
\begin{equation}\label{ineq:l2a5}
\begin{aligned}
&(\div_{S_h}(\nabla_S\phi)^\ele, \div_{S_h}(\nabla_S\psi)^\ele)_{S_h} - (\div_{S_h}G_h\phi_I,\div_{S_h}G_h\psi_I)_{S_h}\\
&\lesssim (\div_{S_h}((\nabla_S\phi)^\ele-G_h\phi_I), (\Delta_S\psi)^\ele)_{S_h}+(\div_{S_h}((\nabla_S\psi)^\ele-G_h\psi_I), (\Delta_S\phi)^\ele)_{S_h}\\
&+h^2\|\phi\|_{3,S}\|\psi\|_{3,S}.
\end{aligned}
\end{equation}
Using the definition of $a_h(\cdot,\cdot)$, Green's formula, and \eqref{ineq:l2a5}, we get:
\begin{equation}\label{ineq:l2a6a}
\begin{aligned}
&(\div_{S_h}(\nabla_S\phi)^\ele, \div_{S_h}(\nabla_S\psi)^\ele)_{S_h} - a_h(\phi_I,\psi_I)\\
&\lesssim (G_h\phi_I-(\nabla_S\phi)^\ele, \nabla_{S_h}(\Delta_S\psi)^\ele)_{S_h} + (G_h\psi_I-(\nabla_S\psi)^\ele, \nabla_{S_h}(\Delta_S\phi)^\ele)_{S_h}\\
&+\sum_{E\in\mathcal E_h}(\avg{(\Delta_S\psi)^\ele},\jp{(\nabla_S\phi)^\ele\cdot\mathbf n_E})_E - \sum_{E\in\mathcal E_h}(\avg{(\Delta_S\psi)^\ele-\div_{S_h}G_h\psi_I},\jp{G_h\phi_I\cdot\mathbf n_E})_E\\
&+\sum_{E\in\mathcal E_h}(\avg{(\Delta_S\phi)^\ele},\jp{(\nabla_S\psi)^\ele\cdot\mathbf n_E})_E - \sum_{E\in\mathcal E_h}(\avg{(\Delta_S\phi)^\ele-\div_{S_h}G_h\phi_I},\jp{G_h\psi_I\cdot\mathbf n_E})_E\\
&-\frac \gamma h\sum_{E\in\mathcal E_h}(\jp{G_h\phi_I\cdot\mathbf n_E},\jp{G_h\psi_I\cdot\mathbf n_E})_E - \gamma_{\text{stab}}\mathfrak S_h(\phi_I,\psi_I) + h^2\|\phi\|_{3,S}\|\psi\|_{3,S}.
\end{aligned}
\end{equation}

Next, we estimate each term in \eqref{ineq:l2a6a} separately. Starting with the term involving $G_h\phi_I-(\nabla_S\phi)^\ele$, we have by \eqref{equ:consistency}:
\begin{equation}\label{ineq:l2a6}
\begin{aligned}
  &(G_h\phi_I-(\nabla_S\phi)^\ele, \nabla_{S_h}(\Delta_S\psi)^\ele)_{S_h} \lesssim h^2\|\phi\|_{3,S}\|\psi\|_{3,S}.
\end{aligned}
\end{equation}
The term $\sum_{E\in\mathcal E_h}(\avg{(\Delta_S\psi)^\ele},\jp{(\nabla_S\phi)^\ele\cdot\mathbf n_E})_E$ can be bounded as follows:
\begin{equation}\label{ineq:l2a7}
  \sum_{E\in\mathcal E_h}(\avg{(\Delta_S\psi)^\ele},\jp{(\nabla_S\phi)^\ele\cdot\mathbf n_E})_E \lesssim h^2\|\phi\|_{4,S}\|\psi\|_{4,S},
\end{equation}
which follows from \eqref{inq:tranh2e}.
By using \eqref{inq:app2} and \eqref{inq:app3}, we have:
\begin{equation}\label{ineq:l2a8}
	\begin{aligned}
  &-\sum_{E\in\mathcal E_h}(\avg{(\Delta_S\psi)^\ele-\div_{S_h}G_h\psi_I},\jp{G_h\phi_I\cdot\mathbf n_E})_E\\
  &\lesssim \Big(\sum_{E\in\mathcal E_h}h\|\avg{(\Delta_S\psi)^\ele-\div_{S_h}G_h\psi_I}\|_{0,E}^2\Big)^{\frac12}
  \Big(\sum_{E\in\mathcal E_h}h^{-1}\|\jp{G_h\phi_I\cdot\mathbf n_E}\|_{0,E}^2\Big)^{\frac12}\\
  &\lesssim h^2\|\phi\|_{3,S}\|\psi\|_{3,S}.
\end{aligned}
\end{equation}
Due to symmetry, similar estimates can be applied to the second, fifth and sixth terms of the right-hand of \eqref{ineq:l2a6a}  yielding:
\begin{align}
  (G_h\psi_I-(\nabla_S\psi)^\ele, \nabla_{S_h}(\Delta_S\phi)^\ele)_{S_h} &\lesssim h^2\|\phi\|_{3,S}\|\psi\|_{3,S},
  \label{ineq:l2a10}\\
  \sum_{E\in\mathcal E_h}(\avg{(\Delta_S\phi)^\ele},\jp{(\nabla_S\psi)^\ele\cdot\mathbf n_E})_E &\lesssim h^2\|\phi\|_{4,S}\|\psi\|_{4,S},
 \label{ineq:l2a11}\\
 -\sum_{E\in\mathcal E_h}(\avg{(\Delta_S\phi)^\ele-\div_{S_h}G_h\phi_I},\jp{G_h\psi_I\cdot\mathbf n_E})_E &\lesssim h^2\|\phi\|_{3,S}\|\psi\|_{3,S}.
  \label{inq:l2a12}
\end{align}

Analogously, for the term $-\frac \gamma h\sum_{E\in\mathcal E_h}(\jp{G_h\phi_I\cdot\mathbf n_E},\jp{G_h\psi_I\cdot\mathbf n_E})_E$, we can bound it by \eqref{inq:app3} and the
Cauchy-Schwarz inequality as follows:
\begin{equation}\label{ineq:l2a14}
  -\frac \gamma h\sum_{E\in\mathcal E_h}(\jp{G_h\phi_I\cdot\mathbf n_E},\jp{G_h\psi_I\cdot\mathbf n_E})_E \lesssim h^2\|\phi\|_{3,S}\|\psi\|_{3,S}.
\end{equation}

Finally, by the Cauchy-Schwarz inequality and an argument analogous to the one in the derivation of \eqref{inq:T45-rephrased}, we have estimate:
\begin{equation}\label{ineq:l2a9}
  -\mathfrak S_h(\phi_I,\psi_I)\leq \Big(\mathfrak S_h(\phi_I,\phi_I)\Big)^{\frac12}
  \Big(\mathfrak S_h(\psi_I,\psi_I)\Big)^{\frac12}\lesssim h^2\|\phi\|_{3,S}\|\psi\|_{3,S}.
\end{equation}

Putting \eqref{ineq:l2a6}--\eqref{ineq:l2a9} into \eqref{ineq:l2a6a}, we obtain the desired result:
\begin{equation}\label{ineq:l2a15}
(\div_{S_h}(\nabla_S\phi)^\ele, \div_{S_h}(\nabla_S\psi)^\ele)_{S_h} - a_h(\phi_I,\psi_I)
\lesssim h^2\|\phi\|_{4,S}\|\psi\|_{4,S}.
\end{equation}
We complete the proof.
\end{proof}

Then, we show the consistency error of the right hand side.
\begin{lemma}\label{lem:l2b}
Let $u$ be the solution of \eqref{eq:weak}, and $u_h$ be the solution of \eqref{equ:weak}. For $g \in L^2(S)$ and $g_h \in L^2(S_h)$ appear following in \eqref{dual:c} and \eqref{dual:d}, respectively, the following error estimates hold:
\begin{align}
|(g_h,u_h)_{S_h}-(g,u_h^\ell)_S| &\lesssim h^2\|f\|_{0,S}\|g\|_{0,S}, \label{ineq:l2b}\\
|(g,u)_S-(g_h,u^\ele)_{S_h}| &\lesssim h^2\|f\|_{0,S}\|g\|_{0,S}. \label{ineq:l2c}
\end{align}
\end{lemma}

\begin{proof}
The proof of this lemma is similar to the proof of
Lemma \ref{lemma:T1}, and we will omit it.
\end{proof}

Now, we are ready to show our error estimate in $L^2$ norm.

\begin{theorem}\label{thm:L2err}
  Let $u\in H^4(S)$ be the solution of \eqref{eq:weak}, and $u_h\in V_h$ be the solution of \eqref{equ:weak}. Then, the following inequality holds:
  \begin{equation}\label{ineq:L2err}
    \|u-u_h^\ell\|_{L^2(S)/\mathbb R}\lesssim h^2\|f\|_{0,S}.
  \end{equation}
\end{theorem}

\begin{proof}
  By selecting $v=u$ in \eqref{dual:c} and $v_h=u_h$ in \eqref{dual:d} while taking into account $\int_Sg\dif\sigma=0$, we obtain:
  \begin{equation}\label{ineq:L2err1}
  \begin{aligned}
  \|u-u_h^\ell\|_{L^2(S)/\mathbb R}^2&=\|u-u_h^\ell-Q_0^S(u-u_h^\ell)\|_{0,S}^2\\
  & =(g,u-u_h^\ell-Q_0^S(u-u_h^\ell))_S=(g,u)_S-(g,u_h^\ell)_S\\
   &=(g,u)_S-(g_h,u_h)_{S_h}+(g_h,u_h)_{S_h}-(g,u_h^\ell)_S\\
   &=[a(u,z)-a_h(u_h,z_h)]+[(g_h,u_h)_{S_h}-(g,u_h^\ell)_S].
  \end{aligned}
  \end{equation}
  Utilizing \eqref{ineq:l2b}, we can establish the bound:
 \begin{equation}\label{ineq:L2err2}
   (g_h,u_h)_{S_h}-(g,u_h^\ell)_S\lesssim h^2\|f\|_{0,S}\|g\|_{0,S}.
 \end{equation}
Next, we proceed to bound the first term of \eqref{ineq:L2err1} as follows:
\begin{equation}\label{ineq:L2err3}
  \begin{aligned}
    &a(u,z)-a_h(u_h,z_h)\\
   = &[a(u,z)-a_h(u_I,z_I)]+[a_h(u_I,z_I)-a_h(u_h,z_h)]\\
    = &[a(u,z)-a_h(u_I,z_I)]+a_h(u_I,z_I-z_h)+a_h(u_I-u_h,z_h)\\
    = &[a(u,z)-a_h(u_I,z_I)]+a_h(u_I,z_I-z_h)+a_h(u_I-u_h,z_I)\\
    &-[a_h(u_I-u_h,z_I-z_h)]\\
    := & I_1+I_2+I_3+I_4.
  \end{aligned}
  \end{equation}
  According to Lemma \ref{lem:l2a} and the regularity estimate \eqref{inq:regularity}, we can deduce:
  \begin{equation}\label{ineq:l2err3a}
    I_1\lesssim h^2\|u\|_{4,S}\|z\|_{4,S}\lesssim h^2\|f\|_{0,S}\|g\|_{0,S}.
  \end{equation}
By utilizing the Cauchy--Schwarz inequality and Lemma \ref{lem:l2a}, we obtain:
\begin{equation}\label{ineq:L2err4}
  \begin{aligned}
    I_2&=[a(u,z)-a_h(u_I,z_h)]+[a_h(u_I,z_I)-a(u,z)]\\
    &\lesssim (g,u)_S-(g_h,u_I)_{S_h}+h^2\|u\|_{4,S}\|z\|_{4,S}.
   \end{aligned}
  \end{equation}
According to \eqref{ineq:l2c} and the interpolation error estimate, we have:
\begin{equation}\label{ineq:L2err5}
  \begin{aligned}
  (g,u)_S-(g_h,u_I)_{S_h}&=(g,u)_S-(g_h,u^\ele)_{S_h}+(g_h,u^\ele-u_I)_{S_h}\\
  &\lesssim h^2\|f\|_{0,S}\|g\|_{0,S},
   \end{aligned}
  \end{equation}
where we have used Lemma \ref{lem:equivalence} and  the fact that $\|g_h\|_{0,S_h}\lesssim \|g\|_{0,S}$ for the last inequality.
Thus, we obtain:
\begin{equation}\label{ineq:l2err6}
  I_2\lesssim h^2\|f\|_{0,S}\|g\|_{0,S}.
\end{equation}
By swapping $u$ and $z$ in $I_2$, we readily deduce:
\begin{equation}\label{ineq:l2err7}
  I_3\lesssim h^2\|f\|_{0,S}\|g\|_{0,S}.
\end{equation}
Furthermore, the following estimate:
    \begin{equation}\label{ineq:l2err8}
    I_4\leq\norm{u_h-u_I} \norm{z_h-z_I}\lesssim h^2\|f\|_{0,S}\|g\|_{0,S}
  \end{equation}
is immediately obtained by utilizing Lemma \ref{lem:super}.
By completing the proof, we conclude that:
\begin{equation}\label{ineq:l2err9}
  \|u-u_h^\ell\|_{L^2(S)/\mathbb R}^2\lesssim h^2\|f\|_{0,S}\|g\|_{0,S}.
\end{equation}
Finally, since $\|u-u_h^\ell\|_{L^2(S)/\mathbb R} = \|g\|_{0,S}$,  to establish the result, we eliminate $\|g\|_{0,S}$ from both sides of \eqref{ineq:l2err9}.
\end{proof}

\subsection{Error estimate on discrete surfaces}
We extend the error estimates derived from Theorem \ref{thm:prior} and Theorem \ref{thm:L2err} on the exact surface to the approximate surface, encompassing scenarios where the exact surface is unknown or for the sake of computational convenience.
\begin{theorem}\label{thm:diserr}
  Let $u\in H^4(S)$ be the solution of \eqref{eq:weak}, and $u_h\in V_h$ be the solution of \eqref{equ:weak}. We have the following error estimates on the approximate surface $S_h$:
  \begin{align}
  \label{ineq:disdelta}  \|(\Delta_Su)^\ele-\div_{S_h}G_hu_h\|_{0,S_h} & \lesssim h\|f\|_{0,S}, \\
  \label{ineq:L2}\|u^\ele-u_h\|_{L^2(S_h)/\mathbb R}&\lesssim h^2\|f\|_{0,S}.
  \end{align}
\end{theorem}

\begin{proof}
Inequality \eqref{ineq:disdelta} has been obtained in the proof of Theorem \ref{thm:prior}.
For estimate \eqref{ineq:L2}, we have by the triangle inequality that:
\begin{equation}\label{ineq:diserr1}
  \begin{split}
    \|u^\ele-u_h\|_{L^2(S_h)/\mathbb R} & \leq  \|u^\ele-u_h-Q_0^S(u-u_h^\ell)\|_{0,S_h}\\
    &\quad +\|Q_0^S(u-u_h^\ell)-Q_0^{S_h}(u^\ele-u_h)\|_{0,S_h} \\
        &:=\sqcap_1+\sqcap_2.
  \end{split}
\end{equation}
By utilizing Theorem \ref{thm:L2err} and Lemma \ref{lem:equivalence}, we deduce that:
\begin{equation}\label{ineq:diserr2}
  \sqcap_1\lesssim h^2\|f\|_{0,S}.
\end{equation}
On the other hand, we have the estimate (see \cite{larsson2017continuous})
\begin{equation}\label{ineq:diserr3}
  \sqcap_2\lesssim h^2(\|u^\ele-u_h\|_{L^2(S_h)/\mathbb R}+\|u\|_{2,S}).
\end{equation}
Finally, applying estimates \eqref{ineq:diserr2}, \eqref{ineq:diserr3} and \eqref{ineq:diserr1}, in conjunction with the regularity estimate \eqref{inq:regularity}, we obtain the desired result and
the proof is completed.
\end{proof}

\section{Numerical Experiments}
\label{sec:num}

In this section, we conduct three tests to validate and verify our theoretical findings. For the first and third test, we employ uniform refinement to generate finer mesh levels, which are then projected onto the exact surface. Throughout the tests, we set the parameter $\gamma$ to 10 for consistency and the parameter $\gamma_{\text{stab}}$ to 1 for stability. To simplify the notation, we use the following symbols:
\[
\begin{array}{ll}
  (D^2e)_0=\|(\Delta_Su)^\ele-\div_{S_h}G_hu_h\|_{0,S_h}, &(De)_0=\|(\nabla_Su)^\ele-\nabla_{S_h}u_h\|_{0,S_h},\\
  (D_re)_0=\|(\nabla_Su)^\ele-G_hu_h\|_{0,S_h}, &e_0=\|u^\ele-u_h\|_{0,S_h}.
\end{array}
\]
In all numerical examples, we will analyze the convergence rates using two typical gradient recovery operators: Weighted Averaging (WA) introduced in Section \ref{sec:fem}  and Parametric Polynomial Preserving Recovery (PPPR) proposed in \cite{dong2020parametric}.

\subsection{Example 1} In this example, we consider a unit sphere centered at the origin as the surface $S$. The approximate surface and the corresponding mesh are illustrated in Figure \ref{circle}. We choose the true solution to be $u=xy$, and the source term $f$ is determined accordingly. We report the numerical results for WA and PPPR in Table \ref{table1} and Table \ref{table2}, respectively.

From Table \ref{table1}, it is evident that, as predicted by Theorem \ref{thm:diserr}, the convergence orders of $(D^2e)_0$ and $e_0$ are $\mathcal O(h)$ and $\mathcal O(h^2)$, respectively. Table \ref{table1} also indicates that $(De)_0$ converges at a rate of $\mathcal O(h)$, and there is a superconvergence phenomenon for $(D_re)_0$. The superconvergence rate of $(D_re)_0$ is related to the mesh quality, as mentioned in \cite{wei2010superconvergence}

Similarly, Table \ref{table2} shows the analogous convergent behavior for PPPR. It is observed that PPPR performs better regarding $(D_re)_0$, with $G_hu_h$ superconverging to $(\nabla_Su)^\ell$ at a rate of $\mathcal O(h^2)$.

\begin{figure}[htbp]
\centering
\includegraphics[width=0.4\textwidth]{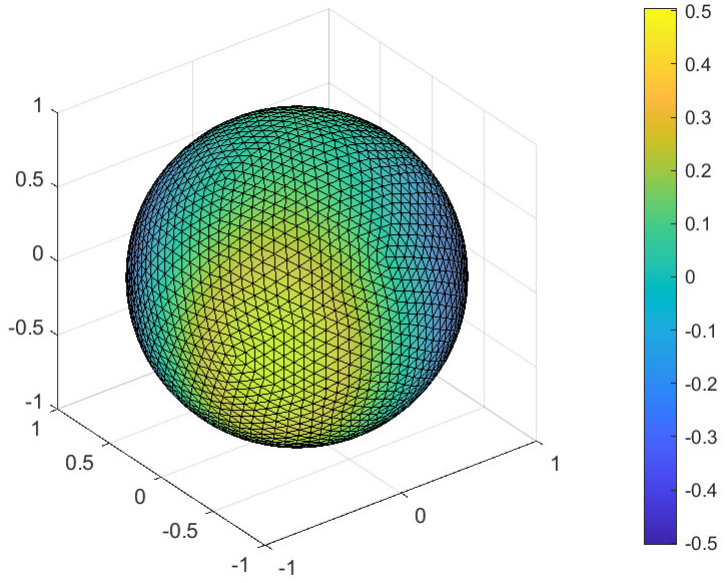}
\caption{The approximate surface and the mesh of Example 1.}
\label{circle}
\end{figure}

\begin{table}[htbp]
  \centering
  \caption{Numerical results of Example 1 with WA.}
 \label{table1}
  \resizebox{\textwidth}{!}
 {
\begin{tabular}{llcccccccc}
  \hline
$G_h$ & Dof & $(D^2e)_0$ & order & $(De)_0$ & order & $e_0$ & order &  $(D_re)_0$ &order\\
\hline
WA&    642   & 4.26e-01 &      &2.21e-01 &      &1.30e-02 &     &4.09e-02  &       \\
WA&   2562   & 2.14e-01 & 0.99 &1.07e-01 & 1.05 &3.32e-03 & 1.98&1.11e-02  & 1.89  \\
WA&  10242   & 1.07e-01 & 1.00 &5.33e-02 & 1.01 &8.33e-04 & 1.99&3.15e-03  & 1.81  \\
WA&  40962   & 5.38e-02 & 1.00 &2.66e-02 & 1.00 &2.08e-04 & 2.00&9.51e-04  & 1.73  \\
WA& 163842   & 2.69e-02 & 1.00 &1.33e-02 & 1.00 &5.21e-05 & 2.00&3.04e-04  & 1.65  \\
 \hline
\end{tabular}
}
\end{table}

\begin{table}[htbp]
  \centering
  \caption{Numerical results of Example 1 with PPPR.}
 \label{table2}
  \resizebox{\textwidth}{!}
 {
\begin{tabular}{llcccccccc}
  \hline
$G_h$ & Dof & $(D^2e)_0$ & order & $(De)_0$ & order & $e_0$ & order &  $(D_re)_0$ &order\\
\hline
PPPR&   642 & 4.59e-01 &      &2.46e-01 &        &9.63e-03 &        &  2.53e-02 &      \\
PPPR&  2562 & 2.29e-01 & 1.00 &1.20e-01 &  1.04  &2.42e-03 & 1.99   &  6.55e-03 & 1.95 \\
PPPR& 10242 & 1.15e-01 & 1.00 &5.91e-02 &  1.02  &6.01e-04 & 2.01   &  1.68e-03 & 1.96 \\
PPPR& 40962 & 5.71e-02 & 1.00 &2.92e-02 &  1.02  &1.48e-04 & 2.02   &  4.30e-04 & 1.97 \\
PPPR&163842 & 2.85e-02 & 1.00 &1.43e-02 &  1.02  &3.63e-05 & 2.03   &  1.09e-04 & 1.98 \\
 \hline
\end{tabular}
}
\end{table}

\subsection{Example 2}

In the second example, we investigate the biharmonic problem on  a torus defined by the signed distance function:
\[d(x,y,z)=\sqrt{(R_0-\sqrt{x^2+y^2})^2+z^2}-R_1,\]
where $R_0=4$ and $R_1=1$. Figure \ref{torus} depicts the approximate surface and its corresponding mesh, which is uniformly derived. The generation of uniform meshes involves mapping uniform meshes from a planar parametric domain onto the torus. The exact solution is chosen as:
\[u=\frac{y}{\sqrt{x^2+y^2}}.\]
Numerical results for the optimal convergence rates of the finite element error are reported in Tables \ref{table3} and \ref{table4}. As observed from these tables, there is an $\mathcal O(h)$ convergence order for $(D^2e)_0$ and $(De)_0$, while an $\mathcal O(h^2)$ convergence order is achieved for $e_0$ and $(D_re)_0$. These results are consistent with the error orders obtained in Example 1. The numerical results validate the conclusions established by Theorem \ref{thm:diserr}.

\begin{figure}[htbp]
\centering
\includegraphics[width=0.4\textwidth]{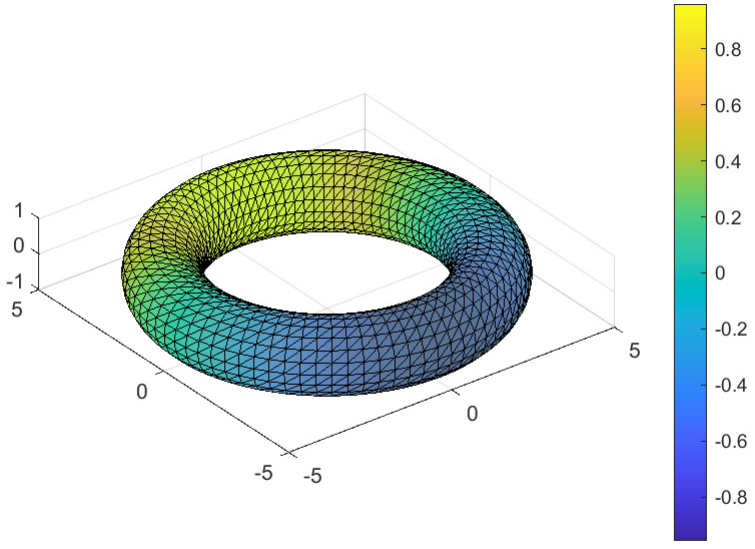}
\caption{The approximate surface and the mesh of Example 2.}
\label{torus}
\end{figure}
\begin{table}[htbp]
  \centering
  \caption{Numerical results of Example 2 with WA.}
 \label{table3}
  \resizebox{\textwidth}{!}
 {
\begin{tabular}{llcccccccc}
  \hline
$G_h$ & Dof & $(D^2e)_0$ & order & $(De)_0$ & order & $e_0$ & order &  $(D_re)_0$ &order\\
\hline
WA&   400   &1.29e-01&      & 4.09e-01&      & 1.45e+00&      & 3.90e-01&      \\
WA&  1600   &5.38e-02& 1.26 & 1.46e-01& 1.49 & 4.40e-01&1.72  & 1.18e-01&1.90  \\
WA&  6400   &2.41e-02& 1.16 & 5.61e-02& 1.38 & 1.15e-01&1.94  & 3.09e-02&1.98  \\
WA& 25600   &1.16e-02& 1.05 & 2.52e-02& 1.16 & 2.91e-02&1.98  & 7.80e-03&2.00  \\
WA&102400   &5.76e-03& 1.01 & 1.22e-02& 1.05 & 7.29e-03&2.00  & 1.96e-03&2.00  \\
 \hline
\end{tabular}
}
\end{table}

\begin{table}[htbp]
  \centering
  \caption{Numerical results of Example 2 with PPPR.}
 \label{table4}
  \resizebox{\textwidth}{!}
 {
\begin{tabular}{llcccccccc}
  \hline
$G_h$ & Dof & $(D^2e)_0$ & order & $(De)_0$ & order & $e_0$ & order &  $(D_re)_0$ &order\\
\hline
PPPR&   400  &  1.28e-01&    &4.24e-01&    &1.52e+00&    &4.05e-01&   \\
PPPR&  1600  &  5.24e-02&1.29&1.44e-01&1.56&4.28e-01&1.83&1.13e-01&1.84 \\
PPPR&  6400  &  2.39e-02&1.14&5.55e-02&1.37&1.11e-01&1.95&2.92e-02&1.96 \\
PPPR& 25600  &  1.16e-02&1.04&2.51e-02&1.14&2.79e-02&1.99&7.37e-03&1.99 \\
PPPR&102400  &  5.76e-03&1.01&1.22e-02&1.04&6.98e-03&2.00&1.85e-03&2.00 \\
 \hline
\end{tabular}
}
\end{table}

\subsection{Example 3}

In our third example, we consider the biharmonic problem \eqref{eq:weak} on a surface determined by the signed distance function:
\[d(x,y,z)=(x-z^2)^2+y^2+z^2-1.\]
Figure \ref{heart} shows the approximate surface and the initial mesh. The exact solution is taken as $u=y$. The numerical results, presented in Tables \ref{table5} and \ref{table6}, reveal similar convergence rates in various norms as observed in Example 1 and Example 2, despite the mesh being less structured in this case.

\begin{figure}[htbp]
\centering
\includegraphics[width=0.4\textwidth]{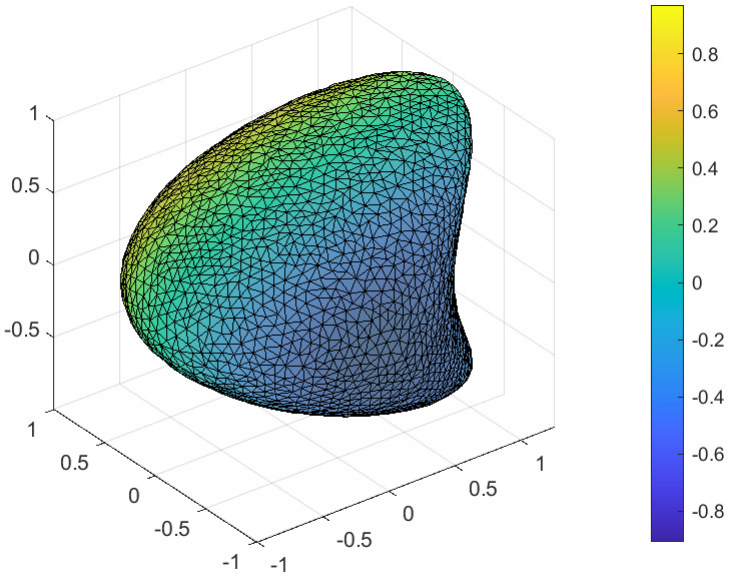}
\caption{The approximate surface and the mesh of Example 3.}
\label{heart}
\end{figure}

\begin{table}[htbp]
  \centering
  \caption{Numerical results of Example 3 with WA.}
 \label{table5}
  \resizebox{\textwidth}{!}
 {
\begin{tabular}{llcccccccc}
  \hline
$G_h$ & Dof & $(D^2e)_0$ & order & $(De)_0$ & order & $e_0$ & order &  $(D_re)_0$ &order\\
\hline
WA&  2697&    9.60e-01&    &   4.85e-01&    &   3.10e-01&    &   3.98493e-01&  \\
WA& 10782&    3.16e-01&1.60&   9.38e-02&2.37&   4.01e-02&2.95&   4.01082e-02&3.31  \\
WA& 43122&    1.63e-01&0.95&   4.04e-02&1.22&   1.07e-02&1.90&   1.12193e-02&1.84  \\
WA&172482&    8.45e-02&0.95&   1.94e-02&1.06&   2.75e-03&1.97&   3.16678e-03&1.82  \\
 \hline
\end{tabular}
}
\end{table}

\begin{table}[htbp]
  \centering
  \caption{Numerical results of Example 3 with PPPR.}
 \label{table6}
  \resizebox{\textwidth}{!}
 {
\begin{tabular}{llcccccccc}
  \hline
$G_h$ & Dof & $(D^2e)_0$ & order & $(De)_0$ & order & $e_0$ & order &  $(D_re)_0$ &order\\
\hline
PPPR&  2697&1.15e-00&    &5.39e-01&    &3.21e-01&    &4.05e-01&\\
PPPR& 10782&4.68e-01&1.30&1.10e-01&2.30&4.24e-02&2.91&4.00e-02&3.34\\
PPPR& 43122&2.08e-01&1.17&4.10e-02&1.42&1.11e-02&1.93&1.03e-02&1.96\\
PPPR&172482&9.77e-02&1.09&1.81e-02&1.18&2.80e-03&1.99&2.57e-03&2.00\\
 \hline
\end{tabular}
}
\end{table}

\section{Conclusion}
\label{sec:con}

In this paper, we present a $C^0$ linear finite element method for solving biharmonic problems on surfaces. Our approach uses only the values at element vertices as degrees of freedom, avoiding the need for complex finite element designs.
We establish rigorous theoretical error estimates, demonstrating optimal convergence rates. Numerical tests confirm the predicted convergence and reveal a superconvergence phenomenon in the recovered gradient.

\section*{Acknowledgments}
This work was supported in part by the Andrew Sisson Fund, Dyason Fellowship, the Faculty Science Researcher Development Grant of the University of Melbourne, the NSFC grant 12131005 and the NSAF grant U2230402.

\bibliographystyle{siamplain}
\bibliography{references}

\begin{thebibliography}{10}

\bibitem{besse2007einstein}
{\sc A.~L. Besse}, {\em Einstein manifolds}, Classics in Mathematics,
  Springer-Verlag, Berlin, 2008.
\newblock Reprint of the 1987 edition.

\bibitem{bonito2011dynamics}
{\sc A.~Bonito, R.~H. Nochetto, and M.~S. Pauletti}, {\em Dynamics of
  biomembranes: effect of the bulk fluid}, Math. Model. Nat. Phenom., 6 (2011),
  pp.~25--43.

\bibitem{brandnernumerical}
{\sc P.~Brandner}, {\em Numerical methods for surface {N}avier-{S}tokes
  equations in stream function formulation}, PhD thesis, Dissertation, RWTH
  Aachen University, 2022.

\bibitem{brenner2003}
{\sc S.~C. Brenner}, {\em Poincar\'{e}-{F}riedrichs inequalities for piecewise
  {$H^1$} functions}, SIAM J. Numer. Anal., 41 (2003), pp.~306--324.

\bibitem{BrSc2008}
{\sc S.~C. Brenner and L.~R. Scott}, {\em The mathematical theory of finite
  element methods}, vol.~15 of Texts in Applied Mathematics, Springer, New
  York, third~ed., 2008.

\bibitem{BS2005}
{\sc S.~C. Brenner and L.-Y. Sung}, {\em {$C^0$} interior penalty methods for
  fourth order elliptic boundary value problems on polygonal domains}, J. Sci.
  Comput., 22/23 (2005), pp.~83--118.

\bibitem{burman2015stabilized}
{\sc E.~Burman, P.~Hansbo, and M.~G. Larson}, {\em A stabilized cut finite
  element method for partial differential equations on surfaces: the
  {L}aplace-{B}eltrami operator}, Comput. Methods Appl. Mech. Engrg., 285
  (2015), pp.~188--207.

\bibitem{burman2019finite}
{\sc E.~Burman, P.~Hansbo, M.~G. Larson, K.~Larsson, and A.~Massing}, {\em
  Finite element approximation of the {L}aplace-{B}eltrami operator on a
  surface with boundary}, Numer. Math., 141 (2019), pp.~141--172.

\bibitem{chenguozhangzou2017}
{\sc H.~Chen, H.~Guo, Z.~Zhang, and Q.~Zou}, {\em A {$C^0$} linear finite
  element method for two fourth-order eigenvalue problems}, IMA J. Numer.
  Anal., 37 (2017), pp.~2120--2138.

\bibitem{Ciarlet2002}
{\sc P.~G. Ciarlet}, {\em The finite element method for elliptic problems},
  vol.~40 of Classics in Applied Mathematics, Society for Industrial and
  Applied Mathematics (SIAM), Philadelphia, PA, 2002.
\newblock Reprint of the 1978 original [North-Holland, Amsterdam; MR0520174 (58
  \#25001)].

\bibitem{dedner2013analysis}
{\sc A.~Dedner, P.~Madhavan, and B.~Stinner}, {\em Analysis of the
  discontinuous {G}alerkin method for elliptic problems on surfaces}, IMA J.
  Numer. Anal., 33 (2013), pp.~952--973.

\bibitem{demlow2007adaptive}
{\sc A.~Demlow and G.~Dziuk}, {\em An adaptive finite element method for the
  {L}aplace-{B}eltrami operator on implicitly defined surfaces}, SIAM J. Numer.
  Anal., 45 (2007), pp.~421--442.

\bibitem{dong2020parametric}
{\sc G.~Dong and H.~Guo}, {\em Parametric polynomial preserving recovery on
  manifolds}, SIAM J. Sci. Comput., 42 (2020), pp.~A1885--A1912.

\bibitem{droniou2019hessian}
{\sc J.~Droniou, B.~P. Lamichhane, and D.~Shylaja}, {\em The {H}essian
  discretisation method for fourth order linear elliptic equations}, J. Sci.
  Comput., 78 (2019), pp.~1405--1437.

\bibitem{DuJu2005}
{\sc Q.~Du and L.~Ju}, {\em Finite volume methods on spheres and spherical
  centroidal {V}oronoi meshes}, SIAM J. Numer. Anal., 43 (2005), pp.~1673--1692
  (electronic).

\bibitem{DuJuTian2011}
{\sc Q.~Du, L.~Ju, and L.~Tian}, {\em Finite element approximation of the
  {C}ahn-{H}illiard equation on surfaces}, Comput. Methods Appl. Mech. Engrg.,
  200 (2011), pp.~2458--2470.

\bibitem{dziuk1988finite}
{\sc G.~Dziuk}, {\em Finite elements for the {B}eltrami operator on arbitrary
  surfaces}, in Partial differential equations and calculus of variations,
  vol.~1357 of Lecture Notes in Math., Springer, Berlin, 1988, pp.~142--155.

\bibitem{dziuk2013finite}
{\sc G.~Dziuk and C.~M. Elliott}, {\em Finite element methods for surface
  {PDE}s}, Acta Numer., 22 (2013), pp.~289--396.

\bibitem{elliott2019secondorder}
{\sc C.~M. Elliott, H.~Fritz, and G.~Hobbs}, {\em Second order splitting for a
  class of fourth order equations}, Math. Comp., 88 (2019), pp.~2605--2634,
  \url{https://doi.org/10.1090/mcom/3425},
  \url{https://doi.org/10.1090/mcom/3425}.

\bibitem{elliott2020splitconstrain}
{\sc C.~M. Elliott and P.~J. Herbert}, {\em Second order splitting of a class
  of fourth order {PDE}s with point constraints}, Math. Comp., 89 (2020),
  pp.~2613--2648, \url{https://doi.org/10.1090/mcom/3556},
  \url{https://doi.org/10.1090/mcom/3556}.

\bibitem{elliott2015evolving}
{\sc C.~M. Elliott and T.~Ranner}, {\em Evolving surface finite element method
  for the {C}ahn-{H}illiard equation}, Numer. Math., 129 (2015), pp.~483--534.

\bibitem{elliott2010modeling}
{\sc C.~M. Elliott and B.~Stinner}, {\em Modeling and computation of two phase
  geometric biomembranes using surface finite elements}, J. Comput. Phys., 229
  (2010), pp.~6585--6612.

\bibitem{EGHLMT2002}
{\sc G.~Engel, K.~Garikipati, T.~J.~R. Hughes, M.~G. Larson, L.~Mazzei, and
  R.~L. Taylor}, {\em Continuous/discontinuous finite element approximations of
  fourth-order elliptic problems in structural and continuum mechanics with
  applications to thin beams and plates, and strain gradient elasticity},
  Comput. Methods Appl. Mech. Engrg., 191 (2002), pp.~3669--3750.

\bibitem{evans2010}
{\sc L.~C. Evans}, {\em Partial differential equations}, vol.~19 of Graduate
  Studies in Mathematics, American Mathematical Society, Providence, RI,
  second~ed., 2010.

\bibitem{gilbarg1998elliptic}
{\sc D.~Gilbarg and N.~S. Trudinger}, {\em Elliptic partial differential
  equations of second order}, Classics in Mathematics, Springer-Verlag, Berlin,
  2001.
\newblock Reprint of the 1998 edition.

\bibitem{guo2018ac}
{\sc H.~Guo, Z.~Zhang, and Q.~Zou}, {\em A {$C^0$} linear finite element method
  for biharmonic problems}, J. Sci. Comput., 74 (2018), pp.~1397--1422.

\bibitem{lamichhane2011stabilized}
{\sc B.~P. Lamichhane}, {\em A stabilized mixed finite element method for the
  biharmonic equation based on biorthogonal systems}, J. Comput. Appl. Math.,
  235 (2011), pp.~5188--5197.

\bibitem{lamichhane2014finite}
{\sc B.~P. Lamichhane}, {\em A finite element method for a biharmonic equation
  based on gradient recovery operators}, BIT, 54 (2014), pp.~469--484.

\bibitem{larsson2017continuous}
{\sc K.~Larsson and M.~G. Larson}, {\em A continuous/discontinuous {G}alerkin
  method and a priori error estimates for the biharmonic problem on surfaces},
  Math. Comp., 86 (2017), pp.~2613--2649.

\bibitem{naga2004posteriori}
{\sc A.~Naga and Z.~Zhang}, {\em A posteriori error estimates based on the
  polynomial preserving recovery}, SIAM J. Numer. Anal., 42 (2004),
  pp.~1780--1800.

\bibitem{naga2005polynomial}
{\sc A.~Naga and Z.~Zhang}, {\em The polynomial-preserving recovery for higher
  order finite element methods in 2{D} and 3{D}}, Discrete Contin. Dyn. Syst.
  Ser. B, 5 (2005), pp.~769--798.

\bibitem{Ni1971}
{\sc J.~Nitsche}, {\em \"uber ein {V}ariationsprinzip zur {L}\"osung von
  {D}irichlet-{P}roblemen bei {V}erwendung von {T}eilr\"aumen, die keinen
  {R}andbedingungen unterworfen sind}, Abh. Math. Sem. Univ. Hamburg, 36
  (1971), pp.~9--15.
\newblock Collection of articles dedicated to Lothar Collatz on his sixtieth
  birthday.

\bibitem{shylaja2020improved}
{\sc D.~Shylaja}, {\em Improved {$L^2$} and {$H^1$} error estimates for the
  {H}essian discretization method}, Numer. Methods Partial Differential
  Equations, 36 (2020), pp.~972--997.

\bibitem{walker2022kirchhoff}
{\sc S.~W. Walker}, {\em The {K}irchhoff plate equation on surfaces: the
  surface {H}ellan-{H}errmann-{J}ohnson method}, IMA J. Numer. Anal., 42
  (2022), pp.~3094--3134.

\bibitem{SW2022}
{\sc S.~W. Walker}, {\em Poincar\'{e} inequality for a mesh-dependent 2-norm on
  piecewise linear surfaces with boundary}, Comput. Methods Appl. Math., 22
  (2022), pp.~227--243.

\bibitem{WSW2024}
{\sc S.~W. Walker}, {\em Approximating the shape operator with the surface
  {H}ellan--{H}errmann--{J}ohnson element}, SIAM J. Sci. Comput., 46 (2024),
  pp.~A1252--A1275.

\bibitem{wei2010superconvergence}
{\sc H.~Wei, L.~Chen, and Y.~Huang}, {\em Superconvergence and gradient
  recovery of linear finite elements for the {L}aplace-{B}eltrami operator on
  general surfaces}, SIAM J. Numer. Anal., 48 (2010), pp.~1920--1943.

\bibitem{Zienkiewicz}
{\sc O.~C. Zienkiewicz and J.~Z. Zhu}, {\em The superconvergent patch recovery
  and a posteriori error estimates. {I}. {T}he recovery technique}, Internat.
  J. Numer. Methods Engrg., 33 (1992), pp.~1331--1364.

\end{thebibliography}

\end{document}